\documentclass[reqno,11pt]{amsart}
\usepackage{amsmath}
\usepackage{amsthm}
\usepackage{epsfig}
\usepackage{psfrag}
\usepackage{graphicx}
\usepackage{graphpap,latexsym,epsf}\graphicspath{{figure/}}
\usepackage{color}
\usepackage{cite}
\usepackage{amssymb,mathrsfs,enumerate}
\usepackage{mathtools}
\usepackage{subfigure}

\usepackage[normalem]{ulem}

\mathtoolsset{showonlyrefs}

\definecolor{citegreen}{rgb}{0,0.6,0}
\definecolor{refred}{rgb}{0.8,0,0}
\usepackage[colorlinks, citecolor=citegreen, linkcolor=refred]{hyperref}

\newcommand{\gtop}{{\gsigma}}
\newcommand{\gsigma}{g{}^{\!\top}}
\newcommand{\gSigma}{\gtop}

{\catcode `\@=11 \global\let\AddToReset=\@addtoreset}
\AddToReset{equation}{section}

\newcounter{mnotecount}[section]

\renewcommand{\themnotecount}{\thesection.\arabic{mnotecount}}

\newcommand{\mnote}[1]
{\protect{\stepcounter{mnotecount}}$^{\mbox{\footnotesize
$
\bullet$\themnotecount}}$ \marginpar{
\raggedright\tiny\em
$\!\!\!\!\!\!\,\bullet$\themnotecount: #1} }

\newcommand{\jj}[1]%
{{\color{red}\mnote{{\color{red}{\bf jj:} #1} }}}

\newcommand{\R}{\mathbb{R}}
\newcommand{\N}{\mathbb{N}}

\def\HHH{{\rm H}}
\def\RRR{{\mathrm R}}

\newcommand{\pa}{\partial}
\newcommand{\Om}{\Omega}
\newcommand{\ffi}{\varphi}

\newcommand{\ep}{\varepsilon}
\newcommand{\rmd}{{\rm d}}

\newcommand{\go}{g}

\newcommand{\Ric}{{\rm Ric}}

\newcommand{\D}{\nabla}
\newcommand{\DD}{{\rm D}^2}
\newcommand{\De}{\Delta}

\newcommand{\na}{\nabla}
\newcommand{\nana}{\nabla^2}

\newcommand{\hhh}{{\rm h}}

\newcommand{\fmax}{f_{{\rm max}}}
\newcommand{\mmax}{m_{{\rm max}}}
\newcommand{\umax}{u_{{\rm max}}}

\mathchardef\emptyset="001F

\definecolor{vgreen}{rgb}{0.1,0.5,0.2}
\definecolor{viola}{RGB}{85,26,139}

\newtheorem{theorem}{Theorem}[section]
\newtheorem{remark}[theorem]{Remark}
\newtheorem{corollary}[theorem]{Corollary}

\newtheorem{proposition}[theorem]{Proposition}

\begin{document}

\hyphenation{ma-ni-fold}

\title
[On the uniqueness of Schwarzschild--de Sitter spacetime]
{On the uniqueness of Schwarzschild--de Sitter spacetime}

\author[S.~Borghini]{Stefano Borghini}
\address{S.~Borghini, Uppsala Universitet, L\"{a}gerhyddsv\"{a}gen 1, 752 37 Uppsala, Sweden}
\email{stefano.borghini@math.uu.se}

\author[P.T.~Chru\'sciel]{Piotr T. Chru\'sciel}
\address{P.T.~Chru\'sciel, University of Vienna, Gravitational Physics
Boltzmanngasse 5, A 1090 Vienna, Austria}
\email{Piotr.Chrusciel@univie.ac.at}


\author[L.~Mazzieri]{Lorenzo Mazzieri}
\address{L.~Mazzieri, Universit\`a degli Studi di Trento, via Sommarive 14, 38123 Povo (TN), Italy}
\email{lorenzo.mazzieri@unitn.it}



\begin{abstract}
We establish a new uniqueness theorem for the three dimensional Schwarzschild--de Sitter metrics. For this some new or improved tools are developed. These include a reverse {\L}ojasiewicz inequality, which  holds in a neighborhood of the extremal points of any smooth function. We further prove smoothness of the set of maxima of the lapse, whenever this set contains a topological hypersurface. This leads to a new strategy for the classification of well behaved static solutions of Einstein equations with a positive cosmological constant, based on the geometry of the maximum-set of the lapse.
\end{abstract}


\maketitle

\smallskip


\smallskip
\noindent

\date{\today}



\section{Introduction}


A basic and fundamental question in the study of the mathematical aspects of General Relativity is the classification of the static solution to the Einstein equations, starting from the case of vacuum solutions.
The first result in this field is the celebrated Israel's Theorem~\cite{Israel}, concerning the uniqueness of the Schwarzschild solution, among the asymptotically flat static solutions to the vacuum Einstein equations. Different proofs and refinements of Israel's result have been proposed by
many authors~\cite{Robinson,Bun_Mas,ZHa_Rob_Sei,Ago_Maz_2,
Reiris,Reiris_claI,Reiris_claII,Chstatic}.

A similar analysis has been performed for asymptotically hyperbolic static solutions. Wang~\cite{Wang_2} and the second author together with Herzlich~\cite{Chr_Her} proved a uniqueness theorem for the Anti de Sitter solution (compare~\cite{Bou_Gib_Hor}), whereas Lee and Neves~\cite{Lee_Nev} obtained a similar result for the Kottler spacetimes with negative mass aspect (compare \cite[Remark~3.4]{GallowayWoolgar}).

In contrast with the significant amount of achievements in the case $\Lambda \leq 0$, very little is known when $\Lambda$ is positive, except in the locally conformally flat case~\cite{Lafontaine}, for perturbations of the model solutions~\cite{Hintz} and under pinching assumptions on the curvature~\cite{Ambrozio}.
 Building on the concept of {\em virtual mass} introduced in~\cite{Bor_Maz_2-I,Bor_Maz_2-II}, we prove the following uniqueness result for the Schwarschild--de Sitter black hole.

\begin{theorem}[Uniqueness of the Schwarzschild--de Sitter Black Hole]
\label{thm:uniqueness}
Let $(M,g)$ be a compact $3$-dimensional totally geodesic spacelike slice
bounded by Killing horizons
 within a $(3+1)$-dimensional static solution to the vacuum Einstein equations with cosmological constant $\Lambda>0$, and let $u \in \mathscr{C}^\infty(M)$ be the corresponding positive lapse function, vanishing on the boundary of $M$.
Assume that the set
$$
{\rm MAX}(u) \,=\,\{ x\in M\,:\,u(x)\,=\, u_{\rm max}\}
$$
is disconnecting the manifold $M$ into an outer region $M_+$ and an inner region $M_-$ having the same {\em virtual mass} $0< m < 1/(3\sqrt{3})$. Then $(M,g)$
can be isometrically embedded in
the Schwarzschild--de Sitter black hole with mass parameter  $m$.
\end{theorem}

For a better understanding of the above statement it is useful to set the stage for the problem. Recall that one can describe an $(n+1)$--dimensional static solution of the vacuum Einstein equations with a positive cosmological constant, by a triple $(M,g,u)$, where $(M,g)$ is an $n$-dimensional Riemannian manifold with boundary, with the {\em lapse function} $u\ge 0$ vanishing precisely  on the boundary of $M$. 
The connected components of the zero-level set of $u$ are referred to as \emph{horizons},  they typically correspond to Killing horizons in the associated spacetime. Depending upon the value assumed by the {\em (normalized) surface gravity} $|\D u|/\umax$, each horizon is either of cosmological type (low surface gravity) or of black hole type (high surface gravity). When the set on which $u$ attains its maximum separates $M$ into two components $M_\pm$, it was shown in \cite{Bor_Maz_2-II,Bor_Maz_2-I} how to assign a {\em virtual mass} $m_\pm$ to each component. The virtual masses are  shown there to obey the inequality $m_+\ge m_-$ when $M_+$ is bounded only by horizons of cosmological type (we then say that $M_+$  is an {\em outer region}).
By definition, for $n=3$, the whole range of the virtual mass parameter $m$ is $0\le  m \le 1/(3\sqrt{3})$.  The case $m=0$  leads to the de Sitter space~\cite[Theorem 2.3]{Bor_Maz_2-I}. The case $m  =1/(3\sqrt{3})$ leads to the Nariai solution.
The known explicit solutions have the property that the virtual masses coincide, i.e., $m_+ =m= m_-$, for some $0<  m < 1/(3\sqrt{3})$. It was shown in~\cite[Theorem 1.9]{Bor_Maz_2-II} that, in three space-dimensions, equality of the masses and connectedness of the part of the horizon bounding $M_+$ implies that the sharp area bound (see~\cite[Theorem 1.4]{Bor_Maz_2-II})
\begin{equation*}
|\pa M_+| \, \leq \, 4\pi r_+^2(m)
\end{equation*}
is saturated and thus $(M,g,u)$ arises from the Schwarzschild-de Sitter spacetime.
Theorem~\ref{thm:uniqueness} can be seen
 as an improvement of this result:
indeed, we are able to remove the hypothesis on the connectedness of the cosmological horizon. Perhaps more significantly, a completely new strategy of proof is used.

At the heart of the new strategy is a detailed analysis of the locus ${\rm MAX} (u)$, established in a fairly large generality and which we believe of independent interest. Our first main result in this context is a \emph{reverse \L ojasiewicz inequality}, Theorem~\ref{thm:rev_loj} below, which provides an estimate on the gradient of a smooth function in terms of its increment  near and away from its maximum set.
 This is an improved version of~\cite[Proposition~2.3]{Bor_Maz_2-II} (see also~\cite[Section~1.1.5]{Bor-thesis}). The result is used to control potential singularities in the function $W$ which appears in the proof the gradient estimates~\eqref{gradest} below and in turns plays a key role in the proof of Theorem~\ref{thm:uniqueness}.
Our next main result is Corollary~\ref{cor:smooth}, which shows that the top stratum  (defined in~\eqref{eq:deco_Loja} below)   of the set of maxima of an analytic function with Laplacian bounded away from zero is a smooth embedded submanifold. This result is crucial for our strategy, as it allows us to invoke the uniqueness part of the Cauchy--Kovalewskaya Theorem to classify the solutions, leading to:

\begin{theorem}[Geometric criterion for isometric embeddings]
 \label{T9VII19.1}
Let $(M,g)$ be a compact $n$-dimensio\-nal,  $n \ge 3$, totally geodesic spacelike slice
bounded by Killing horizons
 within a $(n+1)$-dimensional static solution to the vacuum Einstein equations with cosmological constant $\Lambda>0$.
Let also $u \in \mathscr{C}^\infty(M)$ be the corresponding positive lapse function, vanishing on the boundary of $M$ and let us denote by $\Sigma$ a connected component of the top stratum of ${\rm MAX}(u)$.
If the metric induced on $\Sigma$ by $g$ is Einstein and $\Sigma$  is  totally umbilic, then $(M,g)$ can be isometrically embedded either in a Birmingham-Kottler spacetime or in a Nariai spacetime.
\end{theorem}

To appreciate the above statement,   recall  that  Birmingham-Kottler (BK) metrics can be written in the form
\begin{equation}\label{9VII19.1}
\gamma = -u^2 dt^2 + \frac{dr^2}{u^2} + r^2 g_\Sigma
 \,,
 \quad
  u^2 = 1 -r^2 - \frac{2m}{r^{n-2}}
  \,,
\end{equation}
with $t \in \R$,
where $m$ is a constant and $g_\Sigma$ is a Riemannian metric on a  manifold $\Sigma$ such that
\begin{equation}\label{9VII19.2}
\Ric(g_\Sigma) = (n-2) g_\Sigma
\,.
\end{equation}
Strictly speaking, $\gamma$ is only well defined away from the zeros of $u$, but a standard procedure allows one to extend the metric $\gamma$ across these to obtain a smooth Lorentzian manifold with spatial topology $M=\R\times \Sigma$. The zero-level set of $u$ becomes a collection of Killing horizons in the extended spacetime. It is in fact possible to periodically identify the $\R$ factor of $M$ to obtain a spacetime with a finite number of Killing horizons.
The resulting spacetime manifolds equipped with the metric $\gamma$ of \eqref{9VII19.1} will be referred to as the \emph{BK spacetimes}.
In our terminology, the Schwarzschild-de Sitter spacetimes are  a special case of the BK spacetimes, where the Riemannian manifold $(\Sigma,g_\Sigma)$ is taken to be a sphere with the canonical metric. In order for the BK spacetime to give a static solution, one needs the lapse function to be well defined. It is clear from formula~\eqref{9VII19.1}
that this happens only if the quantity $1-r^2-2mr^{2-n}$ is nonnegative for some positive values of $r$. This limits the choice of the parameter $m$ to the interval $0\leq m<\mmax$, where $\mmax$ is a constant depending on $n$ that can be computed explicitly (see~\eqref{eq:maxmass}). For any $0< m<\mmax$, the BK spacetime is static in the region $r_-(m)<r<r_+(m)$, where $0<r_-(m)<r_+(m)$ are the two positive roots of $1-r^2-2mr^{2-n}$. The corresponding hypersurfaces $\{r=r_-(m)\}$ and $\{r=r_+(m)\}$ are Killing horizons of the BK spacetime.
Recall also that the Nariai spacetime can be defined as the limit as $m\to\mmax$ of the static BK spacetime. An explicit expression for the metric is
\begin{equation}\label{eq:nariai}
\gamma = -\sin^2(r) dt^2 + \frac{1}{n} dr^2 + \frac{n-2}{n} g_\Sigma
 \,,
\end{equation}
where again $g_\Sigma$ is a Riemannian metric on $\Sigma$ satisfying~\eqref{9VII19.2}.

This paper is structured as follows. In Section~\ref{sec:Loj} we will prove the reverse {\L}ojasiewicz inequality (Theorem~\ref{thm:loj_in}). In Section~\ref{sec:MAX} we will analyze the regularity properties of the extremal set of an analytic function with controlled Laplacian. More precisely, in Theorem~\ref{thm:expansion_u} we show how to expand a function in a neighborhood of the set of the maximum (or minimum) points. Building on this, under the hypothesis that the Hessian does not vanish, we show in Theorem~\ref{thm:SigmaC1} and Corollary~\ref{cor:smooth} that the $(n-1)$-dimensional part of the extremal level set is a real analytic hypersurface without boundary.
It is worth remarking that the results in Sections~\ref{sec:Loj} and~\ref{sec:MAX} are not exclusive of the static realm,
but they hold more generally for large classes of real-valued functions. In fact, other recent applications
of these same properties have appeared in~\cite{Bal_Bat_Rib}, where they have been used to study critical metrics
of the volume functional, and in~\cite{Ago_Bor_Maz} where the classical torsion
problem is discussed.

From Section~\ref{geometric criterion} we start focusing exclusively on static solutions. We show that the estimates given by Theorem~\ref{thm:SigmaC1} allow to trigger the Cauchy-Kovalevskaya Theorem~\ref{thm:CK}, leading to a proof of Theorem~\ref{T9VII19.1} (see Theorem~\ref{thm:CKappli}).
Section~\ref{sec:BHU} is devoted for the most part to the statement and proof of Theorem~\ref{thm:main_ineq_SD}, which is a quite general result stating that if an outer region is next to an inner region, then the virtual mass of the outer region is necessarily greater than or equal to the one of the inner region. This theorem is supplemented by a rigidity statement in the case of equality of the virtual masses. In dimension $3$, it is possible to combine this rigidity statement with Theorem~\ref{T9VII19.1}, leading to the proof of
Theorem~\ref{thm:uniqueness}.
Finally, in Section~\ref{12VII19.1} we discuss how to exploit the Cauchy-Kovalevskaya scheme proposed in Section~\ref{geometric criterion} in order to produce local static solutions.

\section{Reverse {\L}ojasiewicz Inequality}
\label{sec:Loj}


Let $(M,g)$ denote a smooth $n$-dimensional Riemannian manifold, possibly with boundary, $n\geq 2$. Given a smooth function $f:M\to\R$, we will denote by $\fmax$ the maximum value of $f$, when achieved, and by
$$
{\rm MAX}(f)\,=\,\{x\in M\,:\,f(x)\,=\,\fmax\}
$$
the set of the maxima of $f$, when nonempty.
We will assume that ${\rm MAX}(f) $ does not meet the boundary of $M$, if there is one.

We start by recalling the following classical result by
{\L}ojasiewicz, concerning the behaviour of an analytic function near a critical point.

\begin{theorem}[{{\L}ojasiewicz inequality~\cite[Th{\'e}or{\`e}me~4]{Lojasiewicz_1},~\cite{Kur_Par}}]
\label{thm:loj_in}
Let $(M,g)$ be a Riemannian manifold and let $f:M\to\R$ be an analytic function. Then for every point $p\in M$ there exists a neighborhood $U_p\ni p$ and real numbers $c_p>0$ and $1\leq\theta_p<2$ such that for every $x\in U_p$ it holds
\begin{equation}
\label{eq:loj_in}
|\na f|^2(x)\,\geq\,c_p\left|f(x)-f(p)\right|^{\theta_p}\,.
\end{equation}
\end{theorem}

\noindent
Let us make some comments on this result.
First, we observe that the above theorem is only relevant when $p$ is a critical point, as otherwise the proof is trivial.
Another observation is that one can always set $c_p=1$ in~\eqref{eq:loj_in}, at the cost of increasing the value of $\theta_p$ and restricting the neighborhood $U_p$. Nevertheless, the inequality is usually stated as in~\eqref{eq:loj_in}, because one often wants to choose the optimal $\theta_p$.
Let us also notice, in particular, that the above result implies the inequality
$$
|\na f|(x)\,\geq\,\left|f(x)-f(p)\right|
$$
 in a neighborhood of any point of our manifold.

The gradient estimate~\eqref{eq:loj_in} has found important applications in the study of gradient flows, as it allows to control the behaviour of the flow near the critical points. The validity of the {\L}ojasiewicz Inequality has been extended to semicontinuous subanalytic functions in~\cite{Bol_Dan_Lew}, and a generalized version of~\eqref{eq:loj_in} has been developed by Kurdyka~\cite{Kurdyka} for larger classes of functions.
An infinite-dimensional version of~\eqref{eq:loj_in} has been proved by Simon~\cite{Simon}, who used it to study the asymptotic behaviour of parabolic equations near critical points. A {\L}ojasiewicz-like inequality for noncompact hypersurfaces has been discussed in~\cite{Col_Min_2}, where it is exploited as the main technical tool to prove the uniqueness of blow-ups of the mean curvature flow. We refer the reader to~\cite{Bol_Dan_Ley_Maz,Col_Min_1,Fee_Mar} and the references therein
for a thorough discussion of various versions of the {\L}ojasiewicz--Simon inequality, as well as for its applications.

On the other hand, to the authors' knowledge, the opposite inequality has not been discussed yet.
In this section we prove an analogous estimate from above of the gradient near the critical points.
Before stating the result, let us make a preliminary observation. Suppose that we are given a Riemannian manifold $(M,g)$ and a function $f\in\mathscr{C}^\infty(M)$. Let $p\in M$ be a critical point of $f$. If we restrict $f$ to a curve $\gamma$ such that $\gamma(0)=p$ and $\dot\gamma(0)=X$, for some unit vector field $X\in T_p M$, we have
$$
f\circ\gamma(t)\,=\,f(p)\,+\,\frac{\nana f(X,X)}{2}\,t^2\,+o(t^2)\,,
$$
from which we compute
\begin{equation}
\label{eq:simple_case}
\frac{\left(\frac{\pa}{\pa t} (f\circ\gamma)_{|_{t=\tau}}\right)^2}{f(p)-f\circ\gamma(\tau)}\,=\,-\,2\,\frac{\big(\nana f(X,X)\big)^2\,\tau^2\,+\,o(\tau^2)}{\nana f(X,X)\,\tau^2\,+\,o(\tau^2)}\,.
\end{equation}
In particular, under the assumption that $\nana f(X,X)\neq 0$ at $p$, we have that the left hand side of~\eqref{eq:simple_case} is locally bounded. As a consequence, we immediately obtain that the inequality
$$
{|\na f|^2}\,\,\leq\,\,c\,|f(p)-f|
$$
holds in a neighborhood of $p$ for some $c\in\R$, provided we assume that $\nana f(p)(X,X)\neq 0$ for every $X\in T_p M$.
The next theorem tells us that a slightly weaker bound is in force at the maximum (or minimum) points of $f$, without any assumptions on the Hessian.

\begin{theorem}[Reverse {\L}ojasiewicz Inequality]
\label{thm:rev_loj}
Let $(M,g)$ be a
Riemannian manifold, let $f:M\to\R$ be a smooth function and let $\Sigma$ be a connected component of ${\rm MAX}(f)$. If $\Sigma$ is compact, then for every $\theta<1$, there exists an open neighborhood $\Omega\supset \Sigma$ and a real number $c>0$ such that for every $x\in \Omega$ it holds
\begin{equation}
\label{eq:rev_loj}
|\na f|^2(x)\,\leq\,c\left[\fmax-f(x)\right]^{\theta}\,.
\end{equation}
\end{theorem}

\begin{proof}
Let us start by defining the function
$$
w\,=\,|\na f|^2-c\,(\fmax-f)^\theta\,,
$$
where $c>0$ is a constant that will be chosen conveniently later. We compute
$$
\na w\,=\,\na|\na f|^2\,+\,c\,\theta\, (\fmax-f)^{-(1-\theta)}\,\na f\,,
$$
and taking the divergence of the above formula
\begin{align*}
\De w\,
&=\,\De|\na f|^2\,+\,c\,\theta \, \frac{\De f}{(\fmax-f)^{1-\theta}}\,+\,c\,\theta\,(1-\theta)\, \frac{|\na f|^2}{(\fmax-f)^{2-\theta}}
\\
&=\,\De|\na f|^2\,+\,c\,\theta \, \frac{\De f}{(\fmax-f)^{1-\theta}}\,+\,c\,\theta\,(1-\theta) \,\frac{w}{(\fmax-f)^{2-\theta}}\,+\,c^2\theta\,(1-\theta)\, \frac{1}{(\fmax-f)^{2-2\theta}}\,,
\end{align*}
where in the second equality we have used $|\na f|^2=w+c\,(\fmax-f)^{\theta}$. We have obtained the following identity
\begin{equation}
\label{eq:ellipticine}
\De w\,-\,c\,\theta\,(1-\theta)\,\frac{1}{(\fmax-f)^{2-\theta}}\,w\,=\,\theta\, F\,\left[\De f\,+\,(1-\theta)\,F\right]\,+\,\De|\na f|^2\,,
\end{equation}
where
$$
F=\frac{c}{(\fmax-f)^{1-\theta}}\,.
$$
Fix now a connected open neighborhood $\Omega$ of $\Sigma$ with smooth boundary $\pa\Omega$. Since $\Sigma$ is compact, we can suppose that $\overline{\Omega}$ is compact as well.
By definition, we have
$$
F\,\geq\,\frac{c}{\max_{\overline{\Omega}}(\fmax-f)^{1-\theta}}\,,\quad\hbox{on } \overline{\Omega}\,.
$$
In particular, increasing the value of $c$ if necessary, we can make $F$ as large as desired.
Since $\De f$ and $\De|\na f|^2$ are continuous and thus bounded in $\overline{\Omega}$, and since $0<\theta<1$, it follows that, for any $c$ big enough, we have
$$
\theta\, F\,[(1-\theta)\,F\,+\,\De f]\,+\,\De|\na f|^2\,\geq\,0
$$
on the whole $\overline{\Omega}$.
For such values of $c$, the right hand side of~\eqref{eq:ellipticine} is nonnegative, that is,
\begin{equation}
\label{eq:max_pr}
\De w-\frac{\theta\,(1-\theta)\,c}{(\fmax-f)^{2-\theta}}\,w\,\geq\,0\,,\quad\hbox{ in }\,\overline{\Omega}\,.
\end{equation}
Notice that the coefficient that multiplies $w$ in~\eqref{eq:max_pr} is negative, as $f\leq\fmax$ and $0<\theta<1$.
Therefore, we can apply the Weak Maximum Principle~\cite[Corollary~3.2]{Gil_Tru}
to $w$ in any open set where $w$ is $\mathscr{C}^2$ -- that is, on any open subset of $\Omega$ that does not intersect $\Sigma$.
For this reason, it is convenient to choose a number $\ep>0$ small enough so that the tubular neighborhood
$$
B_\ep(\Sigma)\,=\,\{x\in M\,:\,d(x,\Sigma)<\ep\}
$$
is contained inside $\Omega$, and consider the set $\Omega_\ep=\Omega\setminus \overline{B_\ep(\Sigma)}$.
Up to increasing the value of $c$ if needed, we can suppose
$$
c\,\geq\,\frac{\max_{\pa\Omega}|\na f|^2}{\min_{\pa\Omega}(\fmax-f)^\theta}\,\geq\,\max_{\pa\Omega}\frac{|\na f|^2}{(\fmax-f)^\theta}\,,
$$
so that $w\leq 0$ on $\pa\Omega$. Now we apply the Weak Maximum Principle to the function $w$ on the open set $\Omega_\ep$, obtaining
$$
w\,\leq\, \max_{\pa \Omega_\ep} (w)\,=\,\max\left\{\max_{\pa\Omega}(w)\,,\,\max_{\pa B_\ep(\Sigma)} (w)\right\}\,\leq\,\max\left\{ 0\,,\,\max_{\pa B_\ep(\Sigma)} (w)\right\}.
$$
Recalling the definition of $w$, taking the limit as $\ep\to 0$, from the continuity of $f$ and $|\na f|$, it follows that
$$
\lim_{\ep\to 0}\max_{\pa B_\ep(\Sigma)} (w)\,=\,0\,,
$$
hence we obtain $w\leq 0$ on $\Omega$. Translating $w$ in terms of $f$, we have obtained that the inequality
$$
|\na f|^2\,\leq\, c \,(\fmax-f)^\theta
$$
holds in $\Omega$, which is the neighborhood of $\Sigma$ that we were looking for.
\end{proof}

\noindent
In particular, we have the following simple refinement:

\begin{corollary}
\label{cor:rev_loj}
Under the hypotheses of Theorem~\ref{thm:rev_loj}, for any $p\in\Sigma$ and any $\alpha<1$, we have
$$
\lim_{f(x)\neq \fmax,\ x\to p}\,\frac{|\na f|^2(x)}{[\fmax-f(x)]^{\alpha}}\,=\,0\,.
$$
\end{corollary}

\begin{proof}
From Theorem~\ref{thm:rev_loj} it follows that we can choose constants $\alpha<\theta<1$ and $c>0$ such that
$$
\frac{|\na f|^2}{[\fmax-f]^\alpha}\,\leq\,
\frac{c\,[\fmax-f]^\theta}{[\fmax-f]^\alpha} \,=\, c \,[\fmax-f]^{\theta-\alpha}\,,
$$
and, since we have chosen $\theta>\alpha$, the right hand side goes to zero as we approach $p$. This proves the claim.
\end{proof}


\section{Regularity of
the extremal level sets}
\label{sec:MAX}


The well known {\L}ojasiewicz Structure Theorem (established
  in~\cite{Lojasiewicz_2}, see also~\cite[Theorem~6.3.3]{Kra_Par}) states that the set of the critical points ${\rm Crit}(f)$ of a real analytic function $f$ is a (possibly disconnected) stratified analytic subvariety whose strata have dimensions between $0$ and $n-1$. In particular, it follows that the set ${\rm MAX}(f)\subseteq{\rm Crit}(f)$ of the maxima of $f$ can be decomposed as follows
\begin{equation}
\label{eq:deco_Loja}
{\rm MAX}(f)\,=\,\Sigma^0\sqcup \Sigma^1\sqcup\dots\sqcup \Sigma^{n-1}\,,
\end{equation}
where $\Sigma^i$ is a finite union of $i$-dimensional real analytic submanifolds, for every $i=0,\dots,n-1$. This means that, given a point $p\in\Sigma^i$, there exists a neighborhood $p\in\Omega\subset M$ and a real analytic diffeomorphism $\phi:\Omega\to\R^n$ such that
$$
\phi(\Omega\cap\Sigma^i)\,=\,L\cap \phi(\Omega)\,,
$$
for some $i$-dimensional linear space $L\subset\R^n$. In particular, the set $\Sigma^{n-1}$ is a real analytic hypersurface and is usually referred to as the {\em top stratum} of ${\rm MAX}(f)$.
%
In this section we show that we can get much more information about the behaviour of our function $f$ around the maximum points that belong to the top stratum.

\begin{theorem}
\label{thm:expansion_u}
Let $(M,g)$ be a real analytic Riemannian manifold, let $f:M\to\R$ be a real analytic function and let $p\in{\rm MAX}(f)$ be a point in the top stratum of ${\rm MAX}(f)$. Let $\Omega$ be a small neighborhood of $p$ such that $\Sigma=\Omega\cap{\rm MAX}(f)$ is contained in the top stratum and $\Omega\setminus\Sigma$ has two connected components $\Omega_+,\Omega_-$. We define the signed distance to $\Sigma$ as
$$
r(x)\,=\,
\begin{cases}
+ \, d(x,\Sigma)\,,   & \text{ if } x\in \overline{\Omega}_+\,,
\\
- \, d(x,\Sigma)\,,  & \text{ if } x\in \overline{\Omega}_-\,.
\end{cases}
$$
Then there is a real analytic chart $(r,\vartheta)=(r,\vartheta^1,\dots,\vartheta^{n-1})$ with respect to which $f$ admits the following expansion:
\begin{equation}
\label{eq:expansion_u_final_-1}
f(r,\vartheta)\,=\,\fmax\,+\,\frac{\De f(0,\vartheta)}{2}\,r^2\,+r^3\,F(r,\vartheta)\,,
\end{equation}
where $F$ is a real analytic function.
If we also assume that $\De f=-\ffi(f)$, then
\begin{multline}
\label{eq:expansion_f_final}
f\,=\,\fmax\,-\,\frac{\ffi(\fmax)}{2}\,r^2\,+\,\frac{\ffi(\fmax)}{6}\,\HHH\,r^3
\\
\,-\,\frac{\ffi(\fmax)}{24}\left[|\hhh|^2\,+\,2\,\HHH^2\,+\,\RRR\,-\,\RRR^{\Sigma}\,-\,\dot\ffi(\fmax)\right]r^4\,+\,r^5\,G(r,\vartheta)\,,
\end{multline}
where $G$ is a real analytic function.
Here we have denoted by $\dot\ffi$ the derivative of $\ffi$ with respect to $f$, by $\HHH,\hhh$ the mean curvature and second fundamental form of $\Sigma$ with respect to the normal pointing towards $\Omega_+$, and by $\RRR,\RRR^\Sigma$ the scalar curvatures of $M$ and $\Sigma$.
\end{theorem}

\begin{remark}
{\rm
We have formulated Theorem~\ref{thm:expansion_u} in the context of real analytic geometry because of our intended application to the classification of static solutions of Einstein equations.
It is, however, clear
from the proof that the conclusions of  Theorem~\ref{thm:expansion_u} remain true for smooth functions on smooth Riemannian manifolds, with the following modifications:
first, one should  assume from the outset that
 $\Sigma$ is a smooth hypersurface (in which case the existence of a decomposition as in~\eqref{eq:deco_Loja} becomes irrelevant); next, the functions $r$, $F$ and $G$ and the chart $(r,\vartheta)$ are smooth but not necessarily analytic.
 \qed
 }
\end{remark}

\begin{proof}
Let $(x^1,\dots,x^n)$ be a chart centered at $p$, with respect to which the metric $g$ and the function $f$ are real analytic.
From the fact that $p$ belongs to the top stratum of ${\rm MAX}(f)$, it follows that we can choose an open neighborhood $\Omega$ of $p$ in $M$, where the signed distance $r(x)$
is a well defined real analytic function (see for instance~\cite{Kra_Par_smoothdist}, where this result is discussed in full detail in the Euclidean space, however the proofs extend with small modifications to the Riemannian setting). More precisely, we have
$$
r=\phi(x^1,\dots,x^n)\,,
$$
where $\phi$ is a real analytic function. Since $r$ is a signed distance function, we have $|\na r|=1$, which implies in particular that one of the partial derivatives of $\phi$ has to be different from zero. Without loss of generality, let us suppose $\pa\phi/\pa x^1\neq 0$ in a small neighborhood $\Om$ of $p$. As a consequence, we have that the function
$$
U:\,\R^{n+1}\,\to\, \R\,,\qquad U(r,x^1,\dots ,x^n)\,=\,r-\phi(x^1,\dots,x^n)\,.
$$
satisfies $\pa U/\pa x^1=-\pa\phi/\pa x^1\neq 0$ in $\Om$. We can then apply the Real Analytic Implicit Function Theorem (see~\cite[Theorem~2.3.5]{Kra_Par}), from which it follows that there exists a real analytic function $u:\R^n\to\R$ such that
$$
U(r,u(r,x^2,\dots,x^n),x^2,\dots, x^n)=0\,.
$$
In other words, the change of coordinates from $(r,x^2,\dots,x^n)$ to $(x^1,\dots,x^n)$, which is obtained setting $x^1=u(r,x^2,\dots,x^n)$, is real analytic. In particular $f$ is a real analytic function also with respect to the chart $(r,\vartheta)$, where we have set $\vartheta=(\vartheta^1,\dots,\vartheta^{n-1})$ with $\vartheta^i=x^{i+1}$ for $i=1,\dots,n-1$. Since $\Sigma$ coincides with the points where $r=0$ inside $\Omega$, we can apply~\cite[Th\'eor\`eme~3.1]{Malgrange} to get
$$
\fmax-f\,=\,r\,A\,,
$$
where $A=A(r,\vartheta)$ is a nonnegative real analytic function. Furthermore, clearly the function $\fmax-f$ achieves its minimum value $0$ when $r=0$, thus
$$
0\,=\,\frac{\pa}{\pa r}(\fmax-f)_{|_{r=0}}\,=\,\frac{\pa}{\pa r}(r\,A)_{|_{r=0}}\,=\,(A\,+\,r\,\pa A/\pa r)_{|_{r=0}}\,=\,A_{|_{r=0}}\,.
$$
In particular, we can apply~\cite[Th\'eor\`eme~3.1]{Malgrange} again and we find:
\begin{equation}
\label{eq:analytic_expansion_2}
\fmax-f\,=\,r^2\,B\,,
\end{equation}
where $B=B(r,\vartheta)$ is a nonnegative real analytic function.
Computing the Laplacian at the points where $r=0$, using the fact that the gradient of $f$ vanishes there, we get:
$$
\De f=g^{\alpha\beta}(\pa^2_{\alpha\beta}f-\Gamma_{\alpha\beta}^\gamma \pa_\gamma f)\,=\,-2\,g^{11}\,B\,=\,-2\,B \, . $$
It follows that we can rewrite~\eqref{eq:analytic_expansion_2} as
\begin{equation}
\label{eq:expansion_u_rewr}
f(r,\vartheta)\,=\,\fmax\,+\,\frac{\De f(0,\vartheta)}{2}\,r^2\,+r^3\,F(r,\vartheta)\,.
\end{equation}
This concludes the first part of the proof.

We now assume that $\De f=-\ffi(f)$ and we use this to gather more information on the real analytic function $F$. Set $\Sigma_\rho=\{r=\rho\}$ and observe that all $\Sigma_\rho$ with $\rho$ small enough are smooth, since $(r,\vartheta)=(r,\vartheta^1,\dots,\vartheta^{n-1})$ is a real analytic chart and $|\na r|=1\neq 0$. In particular, of course, we have $\Sigma_0=\Sigma\cap\Omega$. On each $\Sigma_\rho$, the Laplacian $\Delta f$ of $f$ satisfies the following well known formula
\begin{equation}
\label{eq:lapl_exp_Sigmar}
\De f\,=\,\nana f({\rm n},{\rm n})\,+\,\HHH\,\langle\na f\,|\,{\rm n}\rangle\,+\,\De^{\!\top} f\,,
\end{equation}
where ${\rm n}=\pa/\pa r$ is the $g$-unit normal to $\Sigma_\rho$, $\HHH$ is the mean curvature of $\Sigma_\rho$ with respect to ${\rm n}$ and $\De^{\!\top} f$ is the Laplacian of the restriction of $f$ to $\Sigma_\rho$ with respect to the metric $\gtop $ induced by $g$ on $\Sigma_\rho$. Evaluating ~\eqref{eq:lapl_exp_Sigmar} at $\rho=0$, since $f=\fmax$ and $|\na f|=0$ on $\Sigma_0$, we immediately get
$$
\nana f(\nu,\nu)\,=\,\De f\,,
$$
in agreement with the expansion~\eqref{eq:expansion_u_rewr}.
We now differentiate formula~\eqref{eq:lapl_exp_Sigmar} twice  with respect to $r$, obtaining
\begin{align*}
\frac{\pa \De f}{\pa r}\,&=\,\frac{\pa^3 f}{\pa r^3}\,+\,\HHH\frac{\pa^2 f}{\pa r^2}\,+\,\frac{\pa \HHH}{\pa r}\,\frac{\pa f}{\pa r}\,+\,\frac{\pa}{\pa r}\De^{\!\top} f\,,
\\
\frac{\pa^2 \De f}{\pa r^2}\,&=\,\frac{\pa^4 f}{\pa r^4}\,+\,\HHH\frac{\pa^3 f}{\pa r^3}\,+\,2\,\frac{\pa \HHH}{\pa r}\,\frac{\pa^2 f}{\pa r^2}\,+\,\frac{\pa^2 \HHH}{\pa r^2}\,\frac{\pa f}{\pa r}\,+\,\frac{\pa^2}{\pa r^2}\De^{\!\top} f\,.
\end{align*}
Let us focus first on the terms involving $\De^{\!\top} f$. Calling $\gtop $ the metric induced by $g$ on $\Sigma_\rho$ and $\Gamma^\top$ the Christoffel symbols of $\gtop $, we have
$$
\De^{\!\top}f\,\,=\,\,
(\gtop )^{ij}\,\frac{\pa^2 f}{\pa \vartheta^i\pa \vartheta^j}\,+\,(\gtop )^{ij}\,\,(\Gamma^\top)^{k}_{ij}\,\,\frac{\pa f}{\pa \vartheta^k}\,,
$$
where the indices $i,j,k$ vary between $1$ and $n-1$.
On the other hand, if we assume that $\De f=-\ffi$ where $\ffi$ is a function of $f$, then from~\eqref{eq:expansion_u_rewr} we get
$$
\frac{\pa^2 f}{\pa \vartheta^i\pa \vartheta^j}_{|_{r=0}}\,=\,\frac{\pa^2 f}{\pa r\pa \vartheta^i}_{|_{r=0}}
\,=\,
\frac{\pa^3 f}{\pa r^2 \pa \vartheta^i}_{|_{r=0}}\,=\,
\frac{\pa^3 f}{\pa r \pa \vartheta^i \pa \vartheta^j}_{|_{r=0}}\,=\,\frac{\pa^4 f}{\pa r^2\pa \vartheta^i\pa \vartheta^j}_{|_{r=0}}\,=\,0\,,
$$
for all $i,j=1,\dots,n-1$. From this, it easily follows
$$
\frac{\pa}{\pa r}\De^{\!\top} f_{|_{r=0}}\,=\,\frac{\pa^2}{\pa r^2}\De^{\!\top}f_{|_{r=0}}\,=\,0\,.
$$
Since we also know that $\pa f/\pa r=0$ and $\pa^2 f/\pa r^2=\De f=-\ffi(\fmax)$ on $\Sigma$, from the expansions above we deduce
\begin{align*}
\frac{\pa^3 f}{\pa r^3}_{|_{r=0}}\,&=\,\ffi(\fmax)\,\HHH\,.
\\
\frac{\pa^4 f}{\pa r^4}_{|_{r=0}}\,&=\,2\,\ffi(\fmax)\,\frac{\pa \HHH}{\pa r}_{|_{r=0}}\,-\,\ffi(\fmax)\,\HHH^2\,+\,\ffi(\fmax)\,\dot\ffi(\fmax)\,.
\end{align*}
Furthermore, from~\cite[Lemma~7.6]{Hui_Pol}  we get
$$
\frac{\pa \HHH}{\pa r}_{|_{r=0}}
\,=\,-\,|\hhh|^2\,-\,\Ric(\nu,\nu)\,=\,\frac{1}{2}\left(\RRR^{\Sigma}\,-\,\RRR\,-\,|\hhh|^2\,-\,\HHH^2\right)
\,,
$$
where in the latter equality we have used the Gauss Codazzi equation.
Now that we have computed the third and fourth derivative of $f$, we can use this information to improve~\eqref{eq:expansion_u_rewr} and get the desired expansion of $f$.
\end{proof}

Theorem~\ref{thm:expansion_u} has some interesting consequences. Let us start from the simplest one. From expansion~\eqref{eq:expansion_u_final_-1}, we can compute the explicit formula for the gradient of $f$ as we approach a point $p$ in the top stratum of ${\rm MAX}(f)$ as
\begin{align}
\label{eq:grad_u_nearSigma_N}
\lim_{x\not\in{\rm MAX}(f),\,x\to p}\,\frac{|\na f|^2 (x)}{\fmax-f(x)}\,&=\,
\lim_{x\to p}\frac{(\De f(p))^2\,r^2(x)\,+\,\mathcal{O}(r^3(x))}{-({\De f}(p)/{2})\,r^2(x)\,+\,\mathcal{O}(r^3(x))}\,
=\,-\,2\,\De f(p)\,.
\end{align}
Notice in particular that formula~\eqref{eq:grad_u_nearSigma_N} improves Corollary~\ref{cor:rev_loj} for points in the top stratum of the set of the maxima.

Another useful consequence of Theorem~\ref{thm:expansion_u} is the following regularity result on the top stratum of ${\rm MAX}(f)$.
In order not to overburden notation we will  denote by  $\Sigma$ the top stratum  $\Sigma^{n-1}$ in the  decomposition~\eqref{eq:deco_Loja} of ${\rm MAX}(f)$.

\begin{theorem}
\label{thm:SigmaC1}
Let $(M,g)$ be a real analytic Riemannian manifold, let $f:M\to\R$ be a real analytic function and let $\Sigma$ be the top stratum of ${\rm MAX}(f)$. If  $p\in\overline{\Sigma}$ and $|\nana f|(p)\neq 0$, then $p\in\Sigma$.
\end{theorem}

\begin{proof}
Let $p\in\overline{\Sigma}$ with $|\nana f|(p)\neq 0$, and let $\Omega$ be a small relatively compact open neighborhood of $p$ in $M$.
From 
what has been said
it follows that we can choose $\Om$ small enough so that
\begin{equation}
\label{eq:Sigma1Sigma2..}
\overline{\Sigma}\cap \Om\,=\,\overline{\Sigma}_1\cup\dots\cup\overline{\Sigma}_k
\end{equation}
for some $k\in\N$, where the $\Sigma_i$'s are connected real analytic hypersurfaces contained in the top stratum $\Sigma$ 
 and $p\in\overline{\Sigma}_i$ for all $i=1,\dots,k$.

From expansion~\eqref{eq:expansion_u_final_-1}, it follows that at any point $x\in\Sigma_i$,  with respect to an orthonormal basis $\nu(x),X_1,\dots,X_{n-1}$, where $\nu(x)$ is the unit normal to $\Sigma_i$, the Hessian of $f$ is represented by a matrix of the form
\begin{equation}
\label{eq:hessian_matrix}
\begin{bmatrix}
    \De f(x) & 0      & \cdots  & 0 \\
    0       & 0      & \cdots  & 0 \\
    \vdots  & \vdots & \ddots & \vdots \\
    0       & 0      & \cdots  & 0
\end{bmatrix}\,\,.
\end{equation}
Since we are assuming that $|\nana f|\neq 0$ and that $x$ is a maximum point for $f$, it is clear that $\De f(x)<0$. Using the fact that eigenvalues are continuous (see for instance~\cite[Chapter~2, Theorem~5.1]{Kato}), it follows that
the eigenvalues of the Hessian of $f$ at $p$ are $\De f(p)$ (which is negative by hypothesis), taken with multiplicity one, and $0$, taken with multiplicity $n-1$. In particular, using again the continuity of the eigenvalues, restricting our neighborhood $\Omega$ if necessary, we can suppose that the minimal eigenvalue of $\nana f$ is simple on the whole $\Omega$. Since the function $\Omega\ni x\mapsto\nana f(x)$ is real analytic, it is known that the simple eigenvectors are real analytic in $\Omega$, see for instance the discussion in~\cite[Chapter~2, \S~1]{Kato} or in~\cite[Section~7]{Kur_Pau}, where a much more general statement in discussed. Therefore the vector $\nu$ extends to a real analytic unit-length vector field throughout $\Omega$. In particular, $\nu$ is real analytic on $\overline{\Sigma}\cap\Omega$ and the tangent space $T_p\overline{\Sigma}=\nu_p^{\perp}$ is well defined.

\begin{figure}
 \centering
 \subfigure[Transversal point~\label{fig:sing1}]
   {\includegraphics[scale=1.4]{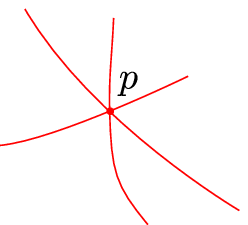}}
 \hspace{15mm}
 \subfigure[Cuspidal point\label{fig:sing2}]
   {\includegraphics[scale=1.4]{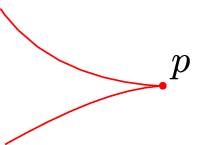}}
   \hspace{15mm}
 \subfigure[Touching point\label{fig:sing3}]{\includegraphics[scale=1.4]{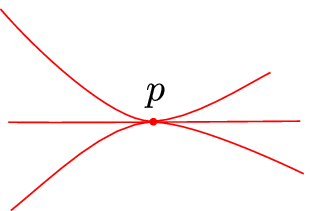}}
 \caption{\small Visual $1$-dimensional representation of the possible singularities of $\overline{\Sigma}$. The first part of the proof of Theorem~\ref{thm:SigmaC1} is concerned with showing that the normal to $\overline{\Sigma}$ is well defined everywhere, thus ruling out transversal singularities like the one pictured in~\ref{fig:sing1}. To exclude cuspidal points~\ref{fig:sing2} or multiple hypersurfaces touching tangentially~\ref{fig:sing3} one needs a different argument, which is presented in the second part of the proof.
 \label{fig:sing_Sigma}}
 \end{figure}

This proves that the case shown in Figure~\ref{fig:sing1} does not occur. However, to prove that $\overline{\Sigma}\cap\Omega$ is a real analytic hypersurface
we still need to exclude the possibility that $p$ is a cuspidal point as in Figure~\ref{fig:sing2}, or that there are multiple hypersurfaces touching at $p$ as in Figure~\ref{fig:sing3}.
We now show that such
singularities cannot happen, by proving that $\Sigma$ is uniformly distant from itself along its normal direction.
For this we start by considering a point $x\in\Sigma\cap\Omega$ and a unit speed geodesic $\gamma_x$ with $\gamma_x(0)=x$ and $\dot\gamma_x(0)$ orthogonal to $\Sigma$.
From the above observations on the Hessian of $f$, it follows that the restriction of $f$ to $\gamma_x$ satisfies the following
$$
(f\circ\gamma_x)(t)\,=\,\fmax\,+\,\frac{\De f(x)}{2}\, t^2\,+\,\sigma_x(t)\,,
$$
where $\sigma_x(t)$ is an error term such that, for small $t$,
$$
|\sigma_x(t)|\,\leq\,\frac{(f\circ\gamma_x)'''(\xi_t)}{3!}\,|t|^3\,,
$$
for some $\xi_t\in\R$ with $|\xi_t|<t$. Since $f$ is a smooth function and $|\dot\gamma_x|\equiv 1$, the derivatives of $f\circ\gamma_x$ are bounded in $\Omega$
 by a constant that does not depend on the choice of $x\in\Sigma$. This means that there exists a constant $C$ such that $|\sigma_x(t)|< C\,|t|^3$ for all $x\in\Sigma$, $t\in\R$ such that $\gamma_x(t)\in\Omega$.
We also notice that the Laplacian of $f$ is strictly negative at $x$, as it follows immediately from the fact that $x$ is a maximum point and the Hessian of $f$ at $p$ does not vanish.
Therefore, for every $|t|\leq |\De f(x)|/(2C)$ it holds
$$
(f\circ\gamma_x)(t)\,=\,\fmax\,+\,\frac{\De f(x)}{2}\, t^2\,+\,\sigma_x(t)\,<\,\fmax\,-\,\left(\frac{|\De f(x)|}{2C}\,-\,|t|\right)C\,t^2\,\leq\,\fmax\,.
$$
This means that every point of $\gamma(t)$ with $|t|\leq |\De f(x)|/(2C)$, $t\neq 0$, does not belong to ${\rm MAX}(f)$. This must hold for every unit speed geodesic starting from a point of $\Sigma$ and directed orthogonally to $\Sigma$. From this property it follows that the pathologies shown in \figurename~\ref{fig:sing_Sigma} do not show up at our point $p$. In fact, if there exist $\Sigma_1,\Sigma_2$ in the decomposition~\eqref{eq:Sigma1Sigma2..} that form a singularity as in \figurename~\ref{fig:sing2} or~\ref{fig:sing3},
 this would mean that orthogonal geodesics starting from points of $\Sigma_1$ arbitrarily close to $p$ would intersect $\Sigma_2$ (which is also contained in ${\rm MAX}(f)$) for arbitrarily small values of $t$. This is in contradiction with what we just proved, hence such singularities cannot exist. This proves that $\overline{\Sigma}\cap\Omega$ is indeed a real analytic hypersurface, possibly with boundary.

 Let us now show that $p\not\in\pa\overline{\Sigma}$.
  We have already seen above that we can define an analytic unit length vector field $\nu$ on the whole neighborhood $\Omega$, in such a way that $\nu$ coincides with the normal vector at each point of $\Sigma\cap\Omega$. From~\cite[Proposizione 3.7.2]{Aba_Tov} it follows that we can find an analytic chart $(x^1,\dots,x^n)$ centered at $p$ such that $\nu=\pa/\pa x^1$ in the whole $\Omega$. Since $\Sigma\cap\Omega$ is smooth and $\nu=\pa/\pa x^1$ is orthogonal to it, it follows that $\overline{\Sigma}\cap\Omega\subseteq\{x^1=0\}$. Since $f$ is analytic and we have shown that $\fmax-f=0$ on an hypersurface contained in $\{x^1=0\}$, it is immediate to deduce that it is possible to factor out $x^1$ from the Taylor expansion of $\fmax-f$. It follows that $\fmax-f=0$ on the whole $\{x^1=0\}$, which implies $\overline{\Sigma}\cap\Omega=\{x^1=0\}$.
 In particular, $p$ necessarily belongs to the interior of $\overline{\Sigma}$.

To conclude that $p$ is contained in the top stratum of ${\rm MAX}(f)$, it remains to show that there are no other lower dimensional components of ${\rm MAX}(f)$ that pass through $p$. In other words, it remains to show that,
making $\Omega$ smaller if necessary, we have $\Omega\cap{\rm MAX}(f)=\Omega\cap\overline{\Sigma}$. To this end, we first observe that,
 since we have already shown that $\overline{\Sigma}$ is a regular analytic hypersurface, the proof of Theorem~\ref{thm:expansion_u} can be repeated without modification to show that formula~\eqref{eq:expansion_u_final_-1} holds in a neighborhood of $p$, where $r$ is the distance from $\overline{\Sigma}$. In particular, the argument used above in this proof tells us that $\overline{\Sigma}\cap\Omega$ is uniformly distant from the rest of ${\rm MAX}(f)$. It is then clear that we can choose $\Omega$ small enough so that $\Omega\cap{\rm MAX}(f)=\Omega\cap\overline{\Sigma}$.
It follows that $\overline{\Sigma}\cap\Omega$ is contained in the top stratum $\Sigma$, which implies in particular that $p\in\Sigma$, as desired.
%
\end{proof}

Theorem~\ref{thm:SigmaC1} tells us that singularities of the $(n-1)$-dimensional part of the set ${\rm MAX}(f)$ can only appear at the points where the Hessian of $f$ vanish. In particular, the following corollary follows at once.

\begin{corollary}
\label{cor:smooth}
Let $(M,g)$ be a real analytic Riemannian manifold, let $f:M\to\R$ be an analytic function and let $\Sigma \neq \emptyset$ be the top stratum of ${\rm MAX}(f)$.
  If $\nana f$ is nowhere vanishing on  $\overline{\Sigma}$, then
$$
 \overline{\Sigma}=\Sigma
  \,,
$$
thus $\Sigma$
is a complete real analytic hypersurface with empty boundary.
\end{corollary}


We emphasize that it is important to require the Hessian of $f$ not to vanish on the whole closure of $\Sigma$. To illustrate this point, consider $f(x,y)=-x^2 y^2$. The function $f$ is clearly analytic on the whole $\R^2$ and satisfies $|\nana f|\neq 0$ at all points of the top stratum $\Sigma=(\{x=0\}\cup\{y=0\})\setminus\{(0,0)\}$ of the set of the maxima. Nevertheless, $\overline{\Sigma}=\{x=0\}\cup\{y=0\}$ is not smooth, which is due to the fact that the Hessian of $f$ vanishes at the singular point $(0,0)=\overline{\Sigma}\setminus\Sigma$.
Another instructive example is provided by the function $f(x,y)=-x^2 y^4$.
Similar examples are easy to construct on compact manifolds $M$ as well: for instance, the same behavior is shown by the function $f(x,y)=-\sin^2 (x)\sin^2 (y)$ on the $2$-torus $\mathbb{T}^2=[0,\pi]^2/\sim$.

\section{A Geometric criterion for the classification of static solutions}
\label{geometric criterion}

This section is devoted to the proof of  Theorem~\ref{T9VII19.1}.
Before discussing it, let us recall briefly, following the setup carefully discussed in~\cite{Bor_Maz_2-II} that a static spacetime with positive cosmological constant arises as a solution of the following problem
\begin{equation}
\label{eq:prob_SD}
\begin{dcases}
u\,\Ric=\DD u+n\,u\,g_0, & \mbox{in } M\\
\ \;\,\De u=-n\, u, & \mbox{in } M\\
\ \ \ \ \; u>0, & \mbox{in }  M\setminus\pa M \\
\ \ \ \ \; u=0, & \mbox{on } \pa M
\end{dcases}
 \qquad  \hbox{with} \  M \ \hbox{compact orientable} \ \hbox{and} \ \RRR\equiv n(n-1)\, .
\end{equation}
We recall from~\cite{Chrusciel_1,MzH} that a static triple $(M,\go,u)$ is necessarily real analytic, so that in particular the results discussed in Section~\ref{sec:MAX} apply. As a consequence, since obviously $\De u=-n\max_M(u)\neq 0$ on ${\rm MAX}(u)$, we can apply Corollary~\ref{cor:smooth} to deduce that the top stratum $\Sigma\subseteq{\rm MAX}(u)$ is a closed real analytic hypersurface. Hence, the second fundamental form, mean curvature and induced metric are all well defined on $\Sigma$, so that the statement of Theorem~\ref{T9VII19.1} is perfectly rigorous.
In the proof, we will need the following version of the Cauchy-Kovalevskaya Theorem for systems of quasilinear PDEs; we report the complete statement for the ease of reference.

\begin{theorem}[Cauchy-Kovalevskaya for nonlinear systems,~\cite{Driver}]
\label{thm:CK}
Let $(x^1,\dots,x^n)$ be a real analytic coordinate chart. Consider a PDE system of the form
\begin{equation}
\label{eq:CK}
\begin{dcases}
\sum_{|I|=k}b_I(x,\pa^1 f,\dots,\pa^{k-1}f)\,\pa^I f\,+\,c(x,\pa^1 f,\dots,\pa^{k-1}f)\,=\,0\,,
\\
\frac{\pa^j f}{(\pa x^1)^j}(x)\,=\,a_j(x)\ \hbox{ for $x\in\{x^1=0\}$ and $j=0,1,\dots,k-1$}
\end{dcases}
\end{equation}
where $I=(I_1,\dots,I_n)$ is a multiindex, $|I|=I_1+\dots+I_n$, $a_j\in\R^m$, $b_I\in\R^{m\times m}$, $c\in\R^{m\times n}$, $f:\R^n\to\R^m$.
Suppose that $a_j$, $b_I$ and $c$ are real analytic functions in their entries and that $b_{(k,0,\dots,0)}$ is an invertible matrix (that is, the hypersurface $\{x^1=0\}$ is noncharacteristic). Then there is a unique real analytic solution to~\eqref{eq:CK}.
\end{theorem}

We are now ready to prove Theorem~\ref{T9VII19.1}, which we rewrite here for the reader's convenience.

\begin{theorem}[Geometric criterion for isometric embeddings]
\label{thm:CKappli}
Let $(M,g)$ be a compact $n$-dimensional,  $n \ge 3$, totally geodesic spacelike slice
bounded by Killing horizons
 within a $(n+1)$-dimensional static solution to the vacuum Einstein equations with positive cosmological constant $\Lambda>0$ and let $\Sigma$ be the top stratum of the set of the maxima of the lapse function.
 Let also $u \in \mathscr{C}^\infty(M)$ be the corresponding positive lapse function, vanishing on the boundary of $M$ and let us denote by $\Sigma$ a connected component of the top stratum of ${\rm MAX}(u)$.
 If the metric induced on $\Sigma$ by $g$ is Einstein and $\Sigma$  is  totally umbilic, then $(M,g)$ can be isometrically embedded either in a Birmingham-Kottler spacetime or in a Nariai spacetime.
 \end{theorem}

\begin{proof}
Let us start by noticing that, given any vector field $X$ on $T\Sigma$ and called $\nu$ the normal to $\Sigma$, it holds
$$
\Ric(\nu,X)\,=\,\frac{\DD u(\nu,X)}{\umax}\,-\,n\,\langle\nu\,|\,X\rangle\,=\,\frac{\DD u(\nu,X)}{\umax}\,=\,0\,,
$$
where the latter equality follows immediately from the general expansion~\eqref{eq:expansion_u_final_-1}. We can then apply the Gauss formula to obtain
$$
0\,=\,\Ric(\nu,X)\,=\,({\rm div}\,\hhh)(X)\,-\,\D_X\HHH
$$
On the other hand, the umbilicity of $\Sigma$ implies $\hhh=\HHH/(n-1)\,g^\Sigma$, hence
$$
0\,=\,({\rm div}\,\hhh)(X)\,-\,\D_X\HHH\,=\,\frac{1}{n-1}\D_X\HHH\,-\,\D_X\HHH\,=\,-\,\frac{n-2}{n-1}\,\D_X\HHH\,.
$$
Since $n\geq 3$, this implies that $\HHH$ is constant. Furthermore, again from expansion~\eqref{eq:expansion_u_final_-1}, together with the static equations~\eqref{eq:prob_SD}, we get
\begin{equation}
\label{eq:riccinunu_zero}
\Ric(\nu,\nu)\,=\,\frac{\DD u(\nu,\nu)}{\umax}\,-\,n\,\langle\nu\,|\,\nu\rangle\,=\,n\,-\,n\,=\,0\,.
\end{equation}
Substituting in the Gauss Codazzi equation:
\begin{align*}
\RRR^\Sigma\,&=\,\RRR\,-\,2\,\Ric(\nu,\nu)\,+\,\HHH^2\,-\,|\hhh|^2
\\
&=\,n(n-1)\,+\,\frac{n-2}{n-1}\,\HHH^2\,.
\end{align*}
Let us now assume $\Ric^\Sigma=(n-2)\lambda g^\Sigma$ for some constant $\lambda\in\R$. Then
$
\RRR^\Sigma\,=\,(n-1)(n-2)\,\lambda
$
and we get
$$
\frac{\HHH}{n-1}\,=\,\sqrt{\lambda\,-\,\frac{n}{n-2}}\,.
$$
In particular, necessarily $\lambda\geq n/(n-2)$ and
\begin{equation}
\label{auxhh}
\hhh\,=\,\sqrt{\lambda - \frac{n}{n-2} }\,g_\Sigma\,.
\end{equation}

We have already mentioned in Theorem~\ref{thm:expansion_u} that the signed distance function
$$
r(x)\,=\,
\begin{dcases}
+d(x,\Sigma)\,, & \hbox{ if } x\in\cap \overline{M}_+\,,
\\
-d(x,\Sigma)\,, & \hbox{ if } x\in\cap\overline{M}_-\,,
\end{dcases}
$$
is an analytic function in a neighborhood $\Omega$ of $\Sigma$. Since clearly $|\D r|\equiv 1$ in $\Omega$, with respect to coordinates $(r,\vartheta)=(r,\vartheta^1,\dots,\vartheta^{n-1})$, with $\vartheta^1,\dots,\vartheta^{n-1}$ coordinates on the hypersurface $\Sigma$, we have
$$
\go\,=\,dr\otimes dr\,+\,g_{ij}\,d\vartheta^i\otimes d\vartheta^j\,,
$$
where the coefficients $g_{ij}$ are functions of the coordinates $(r,\vartheta)$. In particular, with respect to these coordinates, the second fundamental form of a level set of $r$ satisfies
$$
\hhh_{ij}\,=\,\DD_{ij} r\,=\,-\Gamma_{ij}^r\,=\,\frac{1}{2}\,\pa_r g_{ij}\,.
$$
Formula~\eqref{auxhh} tells us that
\begin{equation}
\label{eq:in_con_1}
\frac{\pa g_{ij}}{\pa r}_{\mkern 3mu \vrule height 3ex \mkern3mu {r=0}}\,=\,\,2\,\sqrt{\lambda - \frac{n}{n-2}}\,{g_{ij}^{\Sigma}}_{|_{r=0}}\,,
\end{equation}
We have also the following initial conditions
\begin{equation}
\label{eq:in_con_2}
\begin{aligned}
{\frac{\pa u}{\pa r}}_{\mkern 3mu \vrule height 3ex \mkern3mu {r=0}}\,&=\,0\,,
\\
u_{\mkern 3mu \vrule height 2ex \mkern3mu {r=0}}\,&=\,\max_M(u)\,,
\\
{\go}_{\mkern 3mu \vrule height 2ex \mkern3mu {r=0}}\,&=\,g^\Sigma\,.
\end{aligned}
\end{equation}
We would like to apply the Cauchy-Kovalevskaya Theorem~\ref{thm:CK} to the function
\begin{align*}
f\,:\,\R^n\,&\longrightarrow\,\R^{1+\frac{n(n-1)}{2}}
\\
x=(r,\vartheta)\,&\longmapsto\,\Big(u(x)\,,\,g_{ij}(x)\Big)\,,
\end{align*}
by showing that $f$ satisfies a PDE as in~\eqref{eq:CK} coming from the equations~\eqref{eq:prob_SD} and the initial conditions~\eqref{eq:in_con_1},~\eqref{eq:in_con_2} on $\Sigma=\{r=0\}$.
To this end, let us rewrite problem~\eqref{eq:prob_SD} more explicitly in terms of the derivatives of $u$ and the metric $\go$.
We recall that, in any coordinate system, the components of the Ricci tensor satisfy
\begin{align*}
\RRR_{\alpha\beta}\,
=\,-\frac{g^{\mu\eta}}{2}\left[\pa_{\mu\eta}^2 g_{\alpha\beta}\,
 +\,\pa_{\alpha\beta}^2 g_{\mu\eta}\,-\,\pa_{\mu\alpha}^2 g_{\eta\beta}\,-\,\pa_{\eta\beta}^2 g_{\mu\alpha}\right]\,+\,\hbox{ lower order terms}
\end{align*}
where the lower order terms are polynomial functions  of the components of $g$, the inverse of $g$ and their first derivatives.
In particular, with respect to the coordinates $(r,\vartheta^1,\dots,\vartheta^{n-1})$, the components $\RRR_{ij}$ of the Ricci tensor satisfy
\begin{align*}
\RRR_{ij}\,&=\,-\frac{g^{\mu\eta}}{2}\left[\pa_{\mu\eta}^2 g_{ij}\,+\,\pa_{ij}^2 g_{\mu\eta}\,-\,\pa_{\mu i}^2 g_{\eta j}\,-\,\pa_{\eta j}^2 g_{\mu i}\right]\,+\,\hbox{ lower order terms}
\\
&=\,-\frac{1}{2}\,\pa^2_{rr}g_{ij}-\frac{g^{ab}}{2}\left[\pa_{ab}^2 g_{ij}\,+\,\pa_{ij}^2 g_{ab}\,-\,\pa_{a i}^2 g_{b j}\,-\,\pa_{b j}^2 g_{a i}\right]\,+\,\hbox{ lower order terms}
\end{align*}
where $\mu,\eta$ take the values $r,1,\dots,n-1$ whereas $a,b$ take only the values $1,\dots,n-1$. Again the lower order terms are polynomial functions of the components of $\go$, the inverse of $\go$ and their first derivatives.
Notice that, since $\go$ is nonsingular, the inverse of $\go$ is well defined everywhere and its components $g^{ij}$ are analytic functions of the components $g_{ij}$. Therefore, the lower order terms appearing above are analytic functions of $\go$ and the first derivatives of $\go$.
We can now rewrite problem~\eqref{eq:prob_SD} as
\begin{equation}
\label{eq:prob_rewr}
\begin{dcases}
-\frac{u}{2}\,\pa^2_{rr}g_{ij}-\frac{u}{2}g^{ab}\left[\pa_{ab}^2 g_{ij}\,+\,\pa_{ij}^2 g_{ab}\,-\,\pa_{a i}^2 g_{b j}\,-\,\pa_{b j}^2 g_{a i}\right]\,-\pa^2_{ij}u\,=\,\hbox{lower order terms}\,,
\\
\pa^2_{rr}u\,+\,g^{ab}\,\pa^2_{ab}u\,=\,\hbox{lower order terms}\,,
\end{dcases}
\end{equation}
where the lower order terms appearing in~\eqref{eq:prob_rewr} are analytic functions depending of $u$, $\go$ and their first derivatives.
In order to apply the Cauchy-Kovalevskaya Theorem~\ref{thm:CK} to problem~\eqref{eq:prob_rewr} with initial conditions~\eqref{eq:in_con_1},~\eqref{eq:in_con_2}, it only remains to show that the surface $\Sigma$ is noncharacteristic for the system~\eqref{eq:prob_rewr}. In other words, we have to check that the matrix
\begin{equation}
\label{eq:invertiblematrix}
b_{(2,0,\dots,0)}\,=\,
\begin{bmatrix}
   0      & -u/2 & 0 & \cdots & 0 \\
   0      & 0    & -u/2 & \cdots & 0 \\
   \vdots & \vdots & \vdots & \ddots & \vdots \\
   0 & 0 & 0 & \cdots & -u/2 \\
   1 & 0 & 0 & \cdots & 0
\end{bmatrix}
\end{equation}
is invertible.
Since $u=\max_M(u)> 0$ on $\Sigma$, this is trivially true, hence the Cauchy-Kovalevskaya Theorem can be applied, telling us that there is a unique analytic solution to~\eqref{eq:prob_rewr}. On the other hand, the BK and Nariai spacetimes also solve~\eqref{eq:prob_rewr}. Moreover, it can be checked (see~\eqref{eq:secondfundform_umb} and~\eqref{eq:secondfundform_umb_N} below)  that, for any $\lambda>n/(n-2)$ there is a value $0<m<\mmax$ such that the BK spacetime with mass $m$ satisfies the initial conditions~\eqref{eq:in_con_1},~\eqref{eq:in_con_2} with respect to the chosen $\lambda$, whereas the Nariai solution satisfies the initial conditions~\eqref{eq:in_con_1},~\eqref{eq:in_con_2} with respect to $\lambda=n/(n-2)$. It follows that our solution has to coincide with a model solution inside $\Omega$. In other words, there exists $0<a<\umax$ and an isometric embedding of $(M,\go,u)\cap\{u>a\}$
inside a BK spacetime or inside the Nariai spacetime. If $a\neq 0$, then the matrix~\eqref{eq:invertiblematrix} is invertible on the hypersurface $\{u=a\}$, so we could apply the Cauchy-Kovalevskaya Theorem again on $\{u=a\}$ to extend the isometry. It follows that the isometric embedding can be extended up to the points where $u=0$, that is, up to the horizons.
Therefore, the whole
$(M,\go,u)$ can be isometrically embedded in a BK spacetime or in the Nariai solution.
\end{proof}



\section{Analysis of the interfaces and Black Hole Uniqueness Theorem}
\label{sec:BHU}

This section is mainly devoted to the proof of Theorem~\ref{thm:main_ineq_SD} below, which will be a crucial ingredient, together with Theorem~\ref{thm:CKappli}, in the proof of Theorem~\ref{thm:uniqueness}.

We start with a quick reminder of some definitions that will be helpful in the following discussion. For more details, we refer the reader to~\cite{Bor_Maz_2-II}. A connected component $N$ of $M\setminus{\rm MAX}(u)$ will be called {\em region}. The components of $\pa M$ that belong to the region $N$ are called {\em horizons} of $N$. We will distinguish different types of regions depending on the value of the surface gravity at the horizons of $N$. Namely, a region will be called {\em outer}, {\em inner} or {\em cylindrical}, depending on whether the maximum surface gravity of the horizons is smaller than, greater than or equal to $\sqrt{n}$, respectively.
Given a region $N$ of a solution $(M,\go,u)$ of~\eqref{eq:prob_SD}, the {\em virtual mass} $\mu(N,\go,u)$ is defined as the mass of the model solution that would be responsible for the maximum of the surface gravities detected at the horizons of $N$, see~\cite[Definition~3]{Bor_Maz_2-II}. The virtual mass is always a number between $0$ and $\mmax$, defined as
\begin{equation}
\label{eq:maxmass}
\mmax=\sqrt{\frac{(n-2)^{n-2}}{n^n}}\,.
\end{equation}
It has been proven in~\cite[Theorem~2.3]{Bor_Maz_2-I} that the virtual mass is always well defined and it is equal to zero only on the de Sitter spacetime.
 Finally, given two different regions $A$ and $B$, if $\overline{A}\cap \overline{B}=:\Sigma\neq \emptyset$ then it follows from Corollary~\ref{cor:smooth} that $\Sigma$ is a real analytic hypersurface with empty boundary. The hypersurface $\Sigma$ will be called {\em interface}.

We are now ready to state the following result.

\begin{theorem}[Analysis of the interfaces]
\label{thm:main_ineq_SD}
Let $(M,\go,u)$ be a solution to problem~\eqref{eq:prob_SD}.  Let $A,B$ be two connected components of $M\setminus{\rm MAX}(u)$ such that $\overline{A}\cap\overline{B}=\colon\Sigma\neq\emptyset$.
\begin{itemize}
\item[$(i)$] If $A$ is outer, then $B$ is necessarily inner and
\begin{equation*}
\mu(A,\go,u)\,\geq\,\mu(B,\go,u)\,.
\end{equation*}
Moreover, if $\mu(A,\go,u)=\mu(B,\go,u)=m$ then $\Sigma$ is an umbilical CMC hypersurface and the metric induced by $g$ on $\Sigma$ has constant scalar curvature.

\medskip

\item[$(ii)$] If $A$ is cylindrical, then $B$ is inner or cylindrical and obviously
$$
\mmax\,=\,\mu(A,\go,u)\,\geq\,\mu(B,\go,u)\,.
$$
Moreover, if $\mu(A,\go,u)=\mu(B,\go,u)=\mmax$ (that is, $B$ is cylindrical) then $\Sigma$ is a totally geodesic hypersurface and the metric induced by $g$ on $\Sigma$ has constant scalar curvature.
\end{itemize}

\end{theorem}

In the proof, we will actually compute explicitly the value of second fundamental form, mean curvature and scalar curvature on $\Sigma$ in the rigidity case.
Once Theorem~\ref{thm:main_ineq_SD} is established, Theorem~\ref{thm:uniqueness} follows at once (see Subsection~\ref{sub:uniqueness}), thanks to the geometric criterion established in Theorem~\ref{thm:CKappli}.

\subsection{Analytic preliminaries.}
\label{sub:estimates}

 One of the fundamental ingredients in the proof of Theorem~\ref{thm:main_ineq_SD} is the gradient estimate for the static potential proven in~\cite[Theorem 1.10]{Bor_Maz_2-II}. In Proposition~\ref{pro:gradest} we recall that result and we outline the proof, as it represents an important application of the Reverse {\L}ojasiewicz inequality, and more precisely of
Corollary~\ref{cor:rev_loj}.
Our aim is to show that the gradient of the lapse function $|\D u|_g$ is pointwise controlled by the corresponding quantity $|\D u_m|_{g_m}$ on the model solution of the same mass $m$. To be more precise, one has to use first the Implicit Function Theorem to show that, for any value $m\in(0,\mmax)$ and for any region $N$, there exists a so called {\em pseudo-radial function} $\Psi : N \longrightarrow \R
$, satisfying the relationships
\begin{align}
\label{eq:pr_function}
u(p) &= \sqrt{1-\Psi^2(p)- 2m\Psi^{2-n}(p)} \quad \hbox{for every $p \in N$}
\\
\Psi \equiv r_{\pm}(m) & \quad \hbox{on $N\cap \pa M$} \qquad \hbox{and} \qquad\Psi \equiv r_0(m) \quad\hbox{on $\overline{N}\cap{\rm MAX}(u)$} \, .
\end{align}
Here $r_-(m)<r_+(m)$ are defined as the radii of the two horizons of the BK spacetime of mass $m$, whereas $r_0(m)=[(n-2)m]^{1/n}$ represents the radius of ${\rm MAX}(u)$ in the model solution, and the sign $+$ or $-$ in the boundary condition $\Psi \equiv r_{\pm}(m)$ depends on whether $N$ is an outer or inner region, respectively.
Then, one has to prove the following fundamental gradient estimate:
\begin{proposition}[Gradient Estimate]
\label{pro:gradest}
Let $(M,\go,u)$ be a solution of~\eqref{eq:prob_SD}, and let $N$ be an outer or inner region with virtual mass $m=\mu(N,\go,u)<\mmax$. Then the following inequality holds
\begin{equation}
\label{gradest}
|\D u|_g  \,\, \leq \,\, |\D u_m|_{g_m} \circ \Psi \,,
\end{equation}
where the function $|\D u_m|_{g_m}$ on the right hand side has to be understood as the function of one real variable that associates to some $t \in [r_-(m),r_+(m)]
$
the constant value assumed by the length of the gradient of $u_m$ on the set $\{ |x| = t \}$.
\end{proposition}

\begin{proof}
The argument leading to~\eqref{gradest} is quite delicate and relies on the fact that the quantity
$$
W \, = \,  \frac{\Psi}{|\D u_m|_{g_m} \circ \Psi} \left(|\D u_m|^2_{g_m} \circ \Psi  - |\D u|^2 \right)
$$
satisfies a convenient elliptic partial differential inequality on $N$, namely

\begin{align}
\Delta_{\left({g}/{\Psi^2}\right)} W &- \frac{(n-2)u^2 \Psi^{n-2} + \Psi^n - (n-2)m}{u \,[\Psi^n-(n-2)m]} \,\big\langle \rmd u \, \big| \rmd W \big\rangle_{\!(g/\Psi^2)}
\\
&- n(n-2)m \frac{u^2 \Psi^n }{[\Psi^n-(n-2)m]^2} \left[ \frac{ (n-2) |\D u_m|^2_{g_m} \circ \Psi  + (n+2) |\D u|^2  }{ |\D u_m|^2_{g_m} \circ \Psi } \right] W \, \leq\,  0 \, ,
\end{align}
where $\Delta_{\left({g}/{\Psi^2}\right)}$ and $\langle \cdot | \cdot \rangle_{\left({g}/{\Psi^2}\right)}$ represent the Laplacian and the scalar product of the conformally related metric $\Psi^{-2}g$.
One observes that $W \geq 0$ on $N\cap\pa M$ by construction, since we are comparing with a model solution $(M,g_m,u_m)$ having the same mass as $(N, g, u)$ (see~\cite[Lemma~2.2]{Bor_Maz_2-II} for details). We now want to show that the following limit
$$
\lim_{p\to{\rm MAX}(u)}W(p)\,=\,\lim_{p\to{\rm MAX}(u)} \Psi\,\left(|\D u_m|_{g_m} \circ \Psi  -\frac{|\D u|^2 }{|\D u_m|_{g_m} \circ \Psi} \right)
$$
is equal to zero.
To this end, we first notice that one has an explicit formula for the gradient of the static potential of the model solution, namely
$$
|\D u_m|_{g_m} \circ \Psi\,=\,\Psi\,\left|1-\left(\frac{r_0(m)}{\Psi}\right)^{\!n}\right|\,.
$$
Recalling the definition of $\Psi$, it is easily seen that $|\D u_m|_{g_m}\circ \Psi$ goes to zero at the same rate of $\sqrt{\max_M(u)-u}$ as we approach ${\rm MAX}(u)$.
Therefore, there exists a constant $C>0$ such that
$$
\lim_{p\to{\rm MAX}(u)}W(p)\,=\,-\,C\lim_{p\to{\rm MAX}(u)} \frac{|\D u|^2 }{\sqrt{\umax-u}}\,.
$$
It follows then from the Reverse {\L}ojasiewicz Inequality (more precisely from Corollary~\ref{cor:rev_loj}) that
$W(p) \to 0$, as $p \to {\rm MAX} (u)$.
In particular, applying the Minimum Principle on the region $\Omega^\ep = \{|W| \geq \ep\} \cap N$, for every sufficiently small $\ep>0$, one deduces that $\min_{\Om^\ep} \!W \!\geq - \ep$, and in turn the desired gradient estimate.
\end{proof}

We will also need some estimate for the lapse function and the pseudo-radial function near the interface of two regions. We first focus on the lapse function:

\begin{proposition}
\label{pro:around_Sigma}
Let $(M,\go,u)$ be a solution to problem~\eqref{eq:prob_SD} and let $A,B$ be two connected components of $M\setminus{\rm MAX}(u)$ with $\overline{A}\cap \overline{B}=:\Sigma\neq\emptyset$. Then, the signed distance
\begin{equation}
\label{eq:distance}
r(x)\,=\,
\begin{cases}
+ \, d(x,\Sigma)\,,   & \text{ if } x\in \overline{A}\,,
\\
- \, d(x,\Sigma)\,,  & \text{ if } x\in \overline{B}\,,
\end{cases}
\end{equation}
is an analytic function in a neighborhood of $\Sigma$ and the function $u$ admits the following expansion
\begin{equation}
\label{eq:expansion_u_final}
u\,=\,\max_M(u)\,\left[1-\frac{n}{2}\,r^2\,+\,\frac{n}{6}\,\HHH\,r^3 \,-\,\frac{n}{24}\left(2\,|\mathring{\hhh}|^2\,+\,\frac{n+1}{n-1}\,\HHH^2\,-\,n\right)r^4\,+\mathcal{O}(r^5)\right]\,,
\end{equation}
where $\HHH,\mathring{\hhh}$ are the mean curvature and traceless second fundamental form of $\Sigma$ with respect to the normal $\nu$ pointing towards $A$. 
\end{proposition}

\begin{proof}
From Theorem~\ref{thm:expansion_u} we get the analyticity of the signed distance function $r$ and the following expansion around $\Sigma$:
$$
u\,=\,\max_M(u)\,\left[1-\frac{n}{2}\,r^2\,+\,\frac{n}{6}\,\HHH\,r^3 \,-\,\frac{n}{24}\left(|\hhh|^2\,+\,2\,\HHH^2\,+\,\RRR\,-\,\RRR^\Sigma\,-\,n\right)r^4\,+\mathcal{O}(r^5)\right]\,.
$$
We also know from~\eqref{eq:riccinunu_zero} that $\Ric(\nu,\nu)=0$ on $\Sigma$.
We can then apply the Gauss-Codazzi equation to obtain
\begin{align*}
|\hhh|^2\,+\,2\,\HHH^2\,+\,\RRR\,-\,\RRR^\Sigma\,&=\,2\,\Ric(\nu,\nu)\,+\,2\,|\hhh|^2\,+\,\HHH^2
\\
&
=\,2\,|\mathring{\hhh}|^2\,+\,\frac{2}{n-1}\HHH^2\,+\,\HHH^2
\\
&
=\,2\,|\mathring{\hhh}|^2\,+\,\frac{n+1}{n-1}\,\HHH^2\,.
\end{align*}
Substituting in the expansion for $u$ above, we obtain~\eqref{eq:expansion_u_final}.
\end{proof}

We pass now to the analysis of the pseudo-radial function:

\begin{proposition}
\label{pro:C2}
Let $(M,\go,u)$ be a solution to problem~\eqref{eq:prob_SD}, let $A,B$ be two connected components of $M\setminus{\rm MAX}(u)$ such that $\overline{A}\cap\overline{B}=:\Sigma\neq\emptyset$ and let $r$ be the signed distance to $\Sigma$ defined as in~\eqref{eq:distance}. Let $\Psi:\overline{A}\cup\overline{B}\to\R$ be the pseudo-radial function defined by~\eqref{eq:pr_function} with respect to a parameter $m\in(0,\mmax)$ and with boundary conditions $\Psi=r_+(m)$ on $A\cap\pa M$, $\Psi=r_-(m)$ on $B\cap \pa M$, $\Psi=r_0(m)$ on $\Sigma$. Then the function $\Psi$ is $\mathscr{C}^3$ in a neighborhood of $\Sigma$ and the following expansion holds:
\begin{multline}
\label{eq:expPsi}
\Psi\,=\,\umax(m)\bigg[\frac{r_0(m)}{\umax(m)}\,+\,r\,+\,\frac{n-1}{6}\,K\,r^2\,+\,
\\
\,+\,\frac{1}{12}\left(|\mathring{\hhh}|^2\,-\,2\,n\,+\,\frac{n-1}{9}\,K\,\left((n-4)\frac{\umax(m)}{r_0(m)}-(n+2)\frac{\HHH}{n-1}\right)\right)r^3\,+\,o(r^3)\bigg]\,,
\end{multline}
where
$$
\umax(m)\,=\,\sqrt{1-\left(\frac{m}{\mmax}\right)^{1/n}}\,,\qquad K=\frac{\umax(m)}{r_0(m)}\,-\,\frac{\HHH}{n-1}\,,
$$
and $\HHH$ is the mean curvature of $\Sigma$ with respect to the normal $\pa/\pa r$.
\end{proposition}

\begin{proof}
In order to simplify notations, throughout this proof we will avoid to make the dependence on $m$ explicit. Namely, we will write $\umax$ and $r_0$ instead of $\umax(m)$ and $r_0(m)$.

We first observe that the boundary conditions imposed on $\Psi$ have been chosen in such a way that we can invoke~\cite[Proposition~2.7]{Bor_Maz_2-II}, which tells us that $\Psi$ is $\mathscr{C}^3$ in a neighborhood of $\Sigma$. Therefore, it only remains to compute the explicit expansion of $\Psi$.
We start by recalling that $\Psi=r_0$ on ${\rm MAX}(u)$, and then we write
\begin{equation}
\label{eq:genericexpansion_Psi}
\Psi\,=\,r_0\,+\,v\,r\,+\,w\,r^2\,+\,z\,r^3\,+\,F\,,
\end{equation}
where $v,w,z$ are functions of the coordinates $x^2,\dots,x^n$ only, and $F=o(r^3)$. Now we compute the expansions of the left and right hand sides of the relation $u^2\,=\,1-\Psi^2-2m\Psi^{2-n}$ to obtain information on the functions $v,w,z$.
Taking the square of~\eqref{eq:expansion_u_final}, we get
\begin{equation}
\label{eq:expansion_u2}
u^2\,=\,\umax^2\,\left[1-n\,r^2\,+\,\frac{n}{3}\,\HHH\,r^3 \,+\,\frac{n}{12}\left(4\,n\,-\,2\,|\mathring{\hhh}|^2\,-\,\frac{n+1}{n-1}\,\HHH^2\right)r^4\,+o(r^4)\right]\,,
\end{equation}
On the other hand, with some lenghty (but standard) computations, from~\eqref{eq:genericexpansion_Psi} one obtains
\begin{multline*}
\Psi^2\,=\,r_0^2\bigg[1\,+\,2\,\frac{v}{r_0}\,r\,+\,\left(2\,\frac{w}{r_0}\,+\,\frac{v^2}{r_0}\right)r^2\,+\,2\,\left(\frac{z}{r_0}\,+\,\frac{v\,w}{r_0^2}\right)r^3
\\
\,+\,\left(2\,\frac{v\,z}{r_0^2}\,+\,\frac{w^2}{r_0^2}\right)r^4\,+\,2\,\frac{F}{r_0}\,+\,o(r^4)\bigg]\,,
\end{multline*}
\begin{multline*}
\Psi^{2-n}\,=\,\frac{r_0^2}{m}\bigg[\frac{1}{n-2}-\frac{v}{r_0}\,r\,+\,\left(\frac{n-1}{2}\,\frac{v^2}{r_0^2}\,-\,\frac{w}{r_0}\right)r^2\,
+\,\left(-\frac{n(n-1)}{6}\,\frac{v^3}{r_0^3}\,+\,(n-1)\frac{v\,w}{r_0^2}\,-\,\frac{z}{r_0}\right)r^3
\\
+(n-1)\left(\frac{n(n+1)}{24}\frac{v^4}{r_0^4}\,-\,\frac{n}{2}\,\frac{v^2\,w}{r_0^3}\,+\,\frac{v\,z}{r_0^2}\,+\,\frac{1}{2}\,\frac{w^2}{r_0^2}\right)r^4\,-\,\frac{F}{r_0}\,+\,o(r^4)\,\bigg]\,.
\end{multline*}
From these expansions we get
\begin{multline*}
1-\Psi^2-2m\Psi^{2-n}\,=\,\umax^2\,-\,n\,v^2\,r^2\,+\,n\,\left(\frac{n-1}{3}\,\frac{v^3}{r_0}\,-\,2\,v\,w\right)r^3
\\
+\,n\left(-\frac{(n-1)(n+1)}{12}\frac{v^4}{r_0^2}\,+\,(n-1)\frac{v^2\,w}{r_0}\,-\,2\,v\,z\,-\,w^2\right)r^4\,+\,o(r^4)\,.
\end{multline*}
Comparing with~\eqref{eq:expansion_u2}, we obtain
\begin{align*}
v^2\,&=\,\umax^2\,,
\\
\frac{n-1}{3}\,\frac{v^3}{r_0}\,-\,2\,v\,w\,&=\,\frac{\HHH}{3}\,\umax^2\,,
\\
-\frac{(n-1)(n+1)}{12}\frac{v^4}{r_0^2}\,+\,(n-1)\frac{v^2\,w}{r_0}\,-\,2\,v\,z\,-\,w^2\,&=\,\frac{\umax^2}{12}\left(4\,n\,-\,2\,|\mathring{\hhh}|^2\,-\,\frac{n+1}{n-1}\,\HHH^2\right)\,.
\end{align*}
From the first identity we get $v=\pm\umax$. To decide the sign, we recall the definitions of $\Psi$ and $r$ and we notice that they have been chosen in such a way that $\Psi<r_0$ when $r>0$ and $\Psi>r_0$ when $r<0$. Therefore, recalling~\eqref{eq:genericexpansion_Psi}, the correct choice is to take a positive $v$, hence $v=\umax$. Substituting in the second and third identity, we easily compute the corresponding expressions for $w$ and $z$ and we recover formula~\eqref{eq:expPsi}, as wished.
\end{proof}


\subsection{\texorpdfstring{Proof of Theorem~\ref{thm:main_ineq_SD}--$(i)$}{Proof of Theorem~\ref{thm:main_ineq_SD}--(i)}}

We first focus on the noncylindrical case.
Let $(M,\go,u)$ be a solution to problem~\eqref{eq:prob_SD} and consider a region with virtual mass
$$
m\,=\,\mu(N,\go,u)\,<\,\mmax\,.
$$
We start by rewriting more explicitly the gradient estimate~\eqref{gradest}, by writing
 the value of $|\D u_m|_{g_m} \circ \Psi $ as a function of the pseudo-radial function:
$$
|\D u_m|_{g_m} \circ \Psi \,=\,\Psi\,\big| 1-(n-2)m\Psi^{-n} \big|\,.
$$
In particular, we can rewrite~\eqref{gradest} as
\begin{equation}
\label{gradest2}
\frac{|\D u|^2}{\Psi^2\left[1-(n-2)m\Psi^{2-n}\right]^2}\,\leq\,1\,.
\end{equation}
The aim of this subsection is to compare~\eqref{gradest2} with the expansions discussed in Subsection~\ref{sub:estimates} in order to deduce some consequences on the geometry of the interface $\Sigma$. Specifically, we now prove the following:
\begin{proposition}
\label{pro:Sigma_meancurvature}
Let $(M,\go,u)$ be a solution to problem~\eqref{eq:prob_SD}. Suppose that there are two regions $A,B$ such that $\Sigma:=\overline{A}\cap \overline{B}$ is not empty, and let $m_A,m_B$ be the virtual masses of $A,B$. Let $\HHH$ be the mean curvature of  $\Sigma$ with respect to the normal pointing inside $A$.
\begin{itemize}
\item If $A$ is an outer region, then
\begin{equation*}
\frac{\HHH}{n-1}\,\geq\,\frac{\umax(m_A)}{r_0(m_A)}\,.
\end{equation*}
\smallskip
\item If $B$ is an inner region, then
\begin{equation*}
\frac{\HHH}{n-1}\,\leq\,\frac{\umax(m_B)}{r_0(m_B)}\,.
\end{equation*}
\end{itemize}
\end{proposition}

\begin{proof}
Let $r$ be the signed distance function defined in~\eqref{eq:distance}. We first observe that the metric $\go$ can be written in terms of coordinates $(r,\vartheta^1,\dots,\vartheta^{n-1})$ as
$$
g\,=\,dr\otimes dr\,+\,\gtop _{ij}\,d\vartheta^i\otimes d\vartheta^j\,.
$$
 Starting from~\eqref{eq:expansion_u_final}, one easily computes the following expansion for $|\D u|^2$ along $\Sigma$ as
\begin{align*}
|\D u|^2\,&=\,\left(\frac{\pa u}{\pa r}\right)^2\,+\,(\gtop )^{ij}\frac{\pa u}{\pa \vartheta^i}\,\frac{\pa u}{\pa\vartheta^j}
\\
&=\,n^2\,\umax^2(m)\,r^2\,\left[1\,-\,\HHH\,r\,+\,\mathcal{O}(r^2)\right]\,.
\end{align*}
 We also recall from~\eqref{eq:expPsi} the following expansion:
\begin{align*}
\Psi\,&=\,r_0(m)\,\left[1\,+\,\frac{\umax(m)}{r_0(m)}\,r\,+\,\frac{n-1}{6}\,\frac{\umax(m)}{r_0(m)}\,\left(\frac{\umax(m)}{r_0(m)}-\frac{\HHH}{n-1}\right)\,r^2\,+\,\mathcal{O}(r^3)\right]\,,
\end{align*}
where $\HHH$ is the mean curvature of $\Sigma$ with respect to the unit normal $\nu=\pa/\pa r$ (which is the one pointing inside $A$).
It is then not hard to compute the following expansion
\begin{equation}
\label{eq:naffiexp}
\frac{|\D u|^2}{\Psi^2\left[1-(n-2)m\Psi^{2-n}\right]^2}\,=\,1\,+\,\frac{2(n-1)}{3}\,\left(\frac{\umax(m)}{r_0(m)}-\frac{\HHH}{n-1}\right)\,r\,+\,\mathcal{O}(r^2)\,.
\end{equation}
By definition, $r$ is positive in $A$ and negative in $B$. Comparing with~\eqref{gradest2}, the result follows at once.
\end{proof}


%
%

We are now ready to prove Theorem~\ref{thm:main_ineq_SD} for the non-cylindrical case.
Since $A$ is outer, from Proposition~\ref{pro:Sigma_meancurvature} we get \begin{equation*}
\frac{\HHH}{n-1}\,\geq\,\frac{\umax(m_A)}{r_0(m_A)}\,,
\end{equation*}
where $\HHH$ is the mean curvature of $\Sigma$ with respect to the normal pointing inside $A$. In particular $\HHH$ is positive everywhere on $\Sigma$. If $B$ were also outer, then in the same way we would obtain that the mean curvature of $\Sigma$ with respect to the opposite normal (the one pointing inside $B$) is also positive and this would be a contradiction. Therefore, since we are assuming $m_B<\mmax$, $B$ must be inner. Again from Proposition~\ref{pro:Sigma_meancurvature}, we get
\begin{equation}
\label{eq:meancurv_2bounds}
\frac{\umax(m_A)}{r_0(m_A)}\,\leq\,\frac{\HHH}{n-1}\,\leq\,\frac{\umax(m_B)}{r_0(m_B)}\,.
\end{equation}
Since the function
$$
m\mapsto\frac{\umax(m)}{r_0(m)}\,=\,\sqrt{\frac{1}{r_0^2(m)}\,-\,\frac{1}{r_0^2(\mmax)}}\,=\,\sqrt{\frac{1}{[(n-2)m]^{2/n}}-\frac{n}{n-2}}
$$
is clearly monotonically decreasing, necessarily we must have $m_+\geq m_-$. Furthermore, if $m_+=m=m_-$, then formula~\eqref{eq:meancurv_2bounds} tells us that
\begin{equation}
\label{eq:meancurvconst}
\frac{\HHH}{n-1}\,=\,\frac{\umax(m)}{r_0(m)}\,=\,\sqrt{\frac{1}{r_0^2(m)}\,-\,\frac{1}{r_0^2(\mmax)}}\,.
\end{equation}
Substituting this information inside the expansions for $\Psi$ and $|\D u|^2$, we can refine~\eqref{eq:naffiexp} and compute that, if~\eqref{eq:meancurvconst} holds, then
\begin{equation*}
\frac{|\D u|^2}{\Psi^2\left[1-(n-2)m\Psi^{2-n}\right]^2}\,=\,1\,+\,\frac{1}{2}\,|\mathring{\hhh}|^2\,r^2\,+\,o(r^2)\,.
\end{equation*}
From~\eqref{gradest2} it follows then that $|\mathring{\hhh}|=0$, that is:
\begin{equation}
\label{eq:secondfundform_umb}
\hhh=\frac{\HHH}{n-1}\,g^\Sigma\,=\,\sqrt{\frac{1}{r_0^2(m)}\,-\,\frac{1}{r_0^2(\mmax)}}\,g^\Sigma\,.
\end{equation}
We also know from~\eqref{pro:around_Sigma} that $\Ric(\nu,\nu)=0$ on $\Sigma$. Substituting these pieces of information in the Gauss-Codazzi equation, we get
\begin{align*}
\RRR^\Sigma\,&=\,\RRR\,-\,2\,\Ric(\nu,\nu)\,+\,\HHH^2\,-\,|\hhh|^2
\\
&=\,n(n-1)\,+\,\frac{(n-1)(n-2)}{r_0^2(m)}\,-\,\frac{(n-1)(n-2)}{r_0^2(\mmax)}
\\
&=\,\frac{(n-1)(n-2)}{r_0^2(m)}\,.
\end{align*}
This concludes the proof.

\subsection{\texorpdfstring{Proof of Theorem~\ref{thm:main_ineq_SD}--$(ii)$}{Proof of Theorem~\ref{thm:main_ineq_SD}--(ii)}}
To conclude the proof of Theorem~\ref{thm:main_ineq_SD} it only remains to address the cylindrical case. Given a solution $(M,\go,u)$ of problem~\eqref{eq:prob_SD}, according to~\cite{Bor_Maz_2-II}, we normalize $u$ so that $\max_M(u)=1$.
For a cylindrical region $N\subseteq M\setminus{\rm MAX}(u)$, we have a gradient estimate in the same spirit of~\eqref{gradest}. In fact,~\cite[Proposition~8.2]{Bor_Maz_2-II} tells us that $|\D u|$ is bounded by the norm of the gradient of the static potential of the Nariai solution on the corresponding level set.
More explicitly, we have the following inequality:
\begin{equation}
\label{gradest_nariai}
\frac{|\D u|^2}{n(1-u^2)}\,\leq\,1\,.
\end{equation}
This inequality can then be employed to prove the following analogue of Proposition~\ref{pro:Sigma_meancurvature}
\begin{proposition}
\label{pro:Sigma_meancurvature_N}
Let $(M,\go,u)$ be a solution to problem~\eqref{eq:prob_SD}. Suppose that there are two regions $A,B$ such that $\Sigma:=\overline{A}\cap \overline{B}$ is not empty. Let $\HHH$ be the mean curvature of  $\Sigma$ with respect to the normal pointing inside $A$.
If $A$ is a cylindrical region, then
\begin{equation*}
\HHH\,\geq\,0\,.
\end{equation*}
\end{proposition}

\begin{proof}
The proof of this result is completely analogous to the proof of Proposition~\ref{pro:Sigma_meancurvature}. Since we have normalized $u$ so that its maximum is $1$, from~\eqref{eq:expansion_u_final} we obtain the following expansions in terms of the signed distance $r$:
\begin{align*}
|\D u|^2\,&=\,n^2\,r^2\,\left[1\,-\,\HHH\,r\,+\,\mathcal{O}(r^2)\right]\,,
\\
u\,&=\,\left[1-\frac{n}{2}\,r^2\,+\,\frac{n}{6}\,\HHH\,r^3\,+\mathcal{O}(r^5)\right]\,,
\end{align*}
where $\HHH$ is the mean curvature of $\Sigma$ with respect to the unit normal $\nu=\pa/\pa r$ (which is the one pointing inside $A$). An easy computation now gives
\begin{equation}
\label{eq:naffiexp_N}
\frac{|\D u|^2}{n(1-u^2)}\,=\,1\,-\,\frac{2}{3}\,\HHH\,r\,+\,\mathcal{O}(r^2)\,.
\end{equation}
By definition, $r$ is positive in $A$ and negative in $B$, hence, in order for~\eqref{gradest_nariai} to hold, it must be $\HHH\geq 0$, as wished.
\end{proof}

Now we proceed to the proof of Theorem~\ref{thm:main_ineq_SD} in the cylindrical case.
Since $A$ is cylindrical, from Proposition~\ref{pro:Sigma_meancurvature_N} we get $\HHH\geq 0$,
where $\HHH$ is the mean curvature of $\Sigma$ with respect to the normal pointing inside $A$. In particular $\HHH$ is nonnegative everywhere on $\Sigma$. If $B$ were outer, then Proposition~\ref{pro:Sigma_meancurvature} would tell us that the mean curvature of $\Sigma$ with respect to the opposite normal (the one pointing inside $B$) is positive and this would be a contradiction. Therefore, $B$ must be inner or cylindrical. If $B$ is cylindrical, applying again Proposition~\ref{pro:Sigma_meancurvature_N} to both $A$ and $B$, we get
\begin{equation}
0\,\leq\,\HHH\,\leq\, 0\,.
\end{equation}
Therefore, $\HHH=0$, as wished.
Substituting this information inside the expansions for $u$ and $|\D u|^2$, we can refine~\eqref{eq:naffiexp_N} and compute that, if $\HHH=0$ on $\Sigma$, then
\begin{equation*}
\frac{|\D u|^2}{n(1-u^2)}\,=\,1\,+\,\frac{1}{2}|\mathring{\hhh}|^2\,r^2\,+\,o(r^2)\,.
\end{equation*}
From~\eqref{gradest_nariai} it follows then that $|\mathring{\hhh}|=0$, that is:
\begin{equation}
\label{eq:secondfundform_umb_N}
\hhh=\frac{\HHH}{n-1}\,g^\Sigma\,=\,0\,.
\end{equation}
We also know from~\eqref{pro:around_Sigma} that $\Ric(\nu,\nu)=0$ on $\Sigma$. Substituting these pieces of information in the Gauss-Codazzi equation, we get
\begin{align*}
\RRR^\Sigma\,&=\,\RRR\,-\,2\,\Ric(\nu,\nu)\,+\,\HHH^2\,-\,|\hhh|^2\,=\,n(n-1)\,.
\end{align*}
This concludes the proof.

\subsection{Uniqueness of the Schwarzschild--de Sitter Black Hole.}
\label{sub:uniqueness}

We are now ready to prove Theorem~\ref{thm:uniqueness}, that we restate here for the reader's convenience:

\begin{theorem}
Let $(M,g)$ be a compact $3$-dimensional totally geodesic spacelike slice
bounded by Killing horizons
 within a $(3+1)$-dimensional static solution to the vacuum Einstein equations with positive cosmological constant $\Lambda>0$, and let $u \in \mathscr{C}^\infty(M)$ be the corresponding positive lapse function, vanishing on the boundary of $M$.
Assume that the set
$$
{\rm MAX}(u) \,=\,\{ x\in M\,:\,u(x)\,=\, u_{\rm max}\}
$$
is disconnecting the manifold $M$ into an outer region $M_+$ and an inner region $M_-$ having the same ``virtual mass'' $0< m < 1/(3\sqrt{3})$. Then $(M,g)$
can be isometrically embedded in
the Schwarzschild--de Sitter black hole with mass parameter  $m$.
\end{theorem}
%
%
%
%
%

\begin{proof}
Applying Theorem~\ref{thm:main_ineq_SD} with $A=M_+$, $B=M_-$ and $\Sigma=\overline{M_+}\cap\overline{M_-}$, we deduce that $\Sigma$ is totally umbilical and that the metric $g^\Sigma$ induced by $\go$ on $\Sigma$ has constant positive scalar curvature equal to $2r_0^{-2}(m)$ . Since $\Sigma$ is $2$-dimensional, it follows that $\Sigma$ is isometric to a sphere of radius $r_0(m)$ (in particular, $g^\Sigma$ is Einstein).
We can then apply Theorem~\ref{T9VII19.1} to conclude.
\end{proof}

\section{Local solutions}
 \label{12VII19.1}
A triple $(M,g,u)$ will be said to satisfy the Riemannian static Einstein equations if the spacetime metric
$$
 \gamma_{\mbox{\scriptsize Lor}}: = -u^2 dt^2 + g
 \,,
 \quad
 \partial_t u = 0 = \partial_t g
 \,,
 $$
 satisfies the vacuum Einstein equations, possibly with a cosmological constant.
In this section we will outline how to construct such real analytic triples $(M,g,u)$,
by mimicking the Cauchy problem in general relativity, invoking the Cauchy-Kovalevskaya theorem to justify existence. The construction is a straightforward adaptation of that of Darmois~\cite{Darmois}, a summary of the argument can be found in~\cite{Choquet-BruhatBeginnings}.


Thus, we seek to construct a solution of the static Einstein equations of the form
\begin{equation}\label{13VII19.1}
  g= dr^2 +   \gSigma(r)
  \,,
\end{equation}
where $  \gSigma(r)$ is an $r$-dependent family of metrics on   a, say compact, real analytic manifold $\Sigma$. The initial data at $r=0$ are $u(0)$, $\gSigma(0)$ and their first $r$-derivatives at $r=0$, all taken to be real-analytic.
The Cauchy-Kovalevskaya theorem in Gauss coordinates, as in the Theorem~\ref{thm:CKappli}, provides existence of an interval $r\in (-r_0,r_0)$ and a metric
$$
 \gamma_{\mbox{\scriptsize Lor}} =  dr^2 \underbrace{-u^2 dt^2  +   \gSigma(r)}_{=:\hat g(r)}
 \,,
 $$
 on $\{r\in (-r_0,r_0),t\in\R\}\times \Sigma$.
The metric $ \gamma_{\mbox{\scriptsize Lor}}$ will be $t$-independent and will satisfy the vacuum Einstein equations with a cosmological constant provided that the initial data fields
$$
  \hat g|_{r=0}:= -u(0)^2 dt^2   +   \gSigma(0)
   \quad  \mbox{and} \quad
   \partial_r \hat g|_{r=0}
 $$
 are chosen to be time-independent  and  satisfy the usual general relativistic constraint equations.

%
%

The above provides many new local solutions of the static Einstein equations. Here local refers to the fact, that the solutions might not necessarily extend to boundaries on which $u$ vanishes.

The question then arises, which data on $\Sigma$ leads to manifolds $M$ which are bounded by Killing horizons; equivalently,  manifolds with boundary on which $u$ vanishes. When starting from a critical level set of $u$, a sufficient condition for this is provided by Theorem~\ref{T9VII19.1}.

\bigskip

{\noindent\sc Acknowledgements.}
The authors would like to thank L. Ambrozio, C. Arezzo, A. Carlotto, and C. Cederbaum
for their interest in our work and for stimulating discussions during the preparation of the manuscript. The authors are members of the Gruppo Nazionale per l'Analisi Matematica, la Probabilit\`a e le loro Applicazioni (GNAMPA) of the Istituto Nazionale di Alta Matematica (INdAM) and are partially founded by the GNAMPA Project ``Principi di fattorizzazione, formule di monotonia e disuguaglianze geometriche''. The paper was partially completed during the authors' attendance to the program ``Geometry and relativity'' organized by the Erwin Schr\"{o}dinger International Institute for Mathematics and Physics (ESI).
PTC acknowledges the friendly hospitality and financial support of University of Trento during part of work on this paper. His research was further supported in part by by  the Austrian Science Fund (FWF) under {projects P23719-N16 and P29517-N27, } and by the Polish National Center of Science (NCN) under grant 2016/21/B/ST1/00940.

\bibliographystyle{plain}
\bibliography{biblio}




\end{document}